\documentclass[12pt]{article}
\usepackage{latexsym, amssymb, amsthm, amsmath}
\usepackage{dsfont}

\DeclareMathAlphabet\gothic{U}{euf}{m}{n}

\makeatletter
\def\eqnarray{\stepcounter{equation}\let\@currentlabel=\theequation
\global\@eqnswtrue
\tabskip\@centering\let\\=\@eqncr
$$\halign to \displaywidth\bgroup\hfil\global\@eqcnt\z@
  $\displaystyle\tabskip\z@{##}$&\global\@eqcnt\@ne
  \hfil$\displaystyle{{}##{}}$\hfil
  &\global\@eqcnt\tw@ $\displaystyle{##}$\hfil
  \tabskip\@centering&\llap{##}\tabskip\z@\cr}

\def\endeqnarray{\@@eqncr\egroup
      \global\advance\c@equation\m@ne$$\global\@ignoretrue}

\def\@yeqncr{\@ifnextchar [{\@xeqncr}{\@xeqncr[5pt]}}
\makeatother

\begin{document}
\bibliographystyle{tom}

\newtheorem{lemma}{Lemma}[section]
\newtheorem{thm}[lemma]{Theorem}
\newtheorem{cor}[lemma]{Corollary}
\newtheorem{prop}[lemma]{Proposition}

\theoremstyle{definition}

\newtheorem{remark}[lemma]{Remark}
\newtheorem{exam}[lemma]{Example}
\newtheorem{definition}[lemma]{Definition}

\newcommand{\gota}{\gothic{a}}
\newcommand{\gotb}{\gothic{b}}
\newcommand{\gotc}{\gothic{c}}
\newcommand{\gote}{\gothic{e}}
\newcommand{\gotf}{\gothic{f}}
\newcommand{\gotg}{\gothic{g}}
\newcommand{\gothh}{\gothic{h}}
\newcommand{\gotk}{\gothic{k}}
\newcommand{\gotm}{\gothic{m}}
\newcommand{\gotn}{\gothic{n}}
\newcommand{\gotp}{\gothic{p}}
\newcommand{\gotq}{\gothic{q}}
\newcommand{\gotr}{\gothic{r}}
\newcommand{\gots}{\gothic{s}}
\newcommand{\gott}{\gothic{t}}
\newcommand{\gotu}{\gothic{u}}
\newcommand{\gotv}{\gothic{v}}
\newcommand{\gotw}{\gothic{w}}
\newcommand{\gotz}{\gothic{z}}
\newcommand{\gotA}{\gothic{A}}
\newcommand{\gotB}{\gothic{B}}
\newcommand{\gotG}{\gothic{G}}
\newcommand{\gotL}{\gothic{L}}
\newcommand{\gotS}{\gothic{S}}
\newcommand{\gotT}{\gothic{T}}

\newcounter{teller}
\renewcommand{\theteller}{(\alph{teller})}
\newenvironment{tabel}{\begin{list}%
{\rm  (\alph{teller})\hfill}{\usecounter{teller} \leftmargin=1.1cm
\labelwidth=1.1cm \labelsep=0cm \parsep=0cm}
                      }{\end{list}}

\newcounter{tellerr}
\renewcommand{\thetellerr}{(\roman{tellerr})}
\newenvironment{tabeleq}{\begin{list}%
{\rm  (\roman{tellerr})\hfill}{\usecounter{tellerr} \leftmargin=1.1cm
\labelwidth=1.1cm \labelsep=0cm \parsep=0cm}
                         }{\end{list}}

\newcounter{tellerrr}
\renewcommand{\thetellerrr}{(\Roman{tellerrr})}
\newenvironment{tabelR}{\begin{list}%
{\rm  (\Roman{tellerrr})\hfill}{\usecounter{tellerrr} \leftmargin=1.1cm
\labelwidth=1.1cm \labelsep=0cm \parsep=0cm}
                         }{\end{list}}

\newcounter{proofstep}
\newcommand{\nextstep}{\refstepcounter{proofstep}\vertspace \par 
          \noindent{\bf Step \theproofstep} \hspace{5pt}}
\newcommand{\firststep}{\setcounter{proofstep}{0}\nextstep}

\numberwithin{equation}{section}

\newcommand{\Ni}{\mathds{N}}
\newcommand{\Qi}{\mathds{Q}}
\newcommand{\Ri}{\mathds{R}}
\newcommand{\Ci}{\mathds{C}}
\newcommand{\Ti}{\mathds{T}}
\newcommand{\Zi}{\mathds{Z}}
\newcommand{\Fi}{\mathds{F}}
\newcommand{\Ki}{\mathds{K}}

\renewcommand{\proofname}{{\bf Proof}}

\newcommand{\simh}{{\stackrel{{\rm cap}}{\sim}}}
\newcommand{\ad}{{\mathop{\rm ad}}}
\newcommand{\Ad}{{\mathop{\rm Ad}}}
\newcommand{\alg}{{\mathop{\rm alg}}}
\newcommand{\clalg}{{\mathop{\overline{\rm alg}}}}
\newcommand{\Aut}{\mathop{\rm Aut}}
\newcommand{\arccot}{\mathop{\rm arccot}}
\newcommand{\capp}{{\mathop{\rm cap}}}
\newcommand{\rcapp}{{\mathop{\rm rcap}}}
\newcommand{\diam}{\mathop{\rm diam}}
\newcommand{\divv}{\mathop{\rm div}}
\newcommand{\dom}{\mathop{\rm dom}}
\newcommand{\codim}{\mathop{\rm codim}}
\newcommand{\RRe}{\mathop{\rm Re}}
\newcommand{\IIm}{\mathop{\rm Im}}
\newcommand{\tr}{{\mathop{\rm Tr \,}}}
\newcommand{\Tr}{{\mathop{\rm Tr \,}}}
\newcommand{\Vol}{{\mathop{\rm Vol}}}
\newcommand{\card}{{\mathop{\rm card}}}
\newcommand{\rank}{\mathop{\rm rank}}
\newcommand{\supp}{\mathop{\rm supp}}
\newcommand{\sgn}{\mathop{\rm sgn}}
\newcommand{\essinf}{\mathop{\rm ess\,inf}}
\newcommand{\esssup}{\mathop{\rm ess\,sup}}
\newcommand{\Int}{\mathop{\rm Int}}
\newcommand{\lcm}{\mathop{\rm lcm}}
\newcommand{\loc}{{\rm loc}}
\newcommand{\HS}{{\rm HS}}
\newcommand{\Trn}{{\rm Tr}}
\newcommand{\n}{{\rm N}}
\newcommand{\WOT}{{\rm WOT}}

\newcommand{\at}{@}

\newcommand{\spann}{\mathop{\rm span}}
\newcommand{\one}{\mathds{1}}

\hyphenation{groups}
\hyphenation{unitary}

\renewcommand{\tfrac}[2]{{\textstyle \frac{#1}{#2}}}

\newcommand{\ca}{{\cal A}}
\newcommand{\cb}{{\cal B}}
\newcommand{\cc}{{\cal C}}
\newcommand{\cd}{{\cal D}}
\newcommand{\ce}{{\cal E}}
\newcommand{\cf}{{\cal F}}
\newcommand{\ch}{{\cal H}}
\newcommand{\chs}{{\cal HS}}
\newcommand{\ci}{{\cal I}}
\newcommand{\ck}{{\cal K}}
\newcommand{\cl}{{\cal L}}
\newcommand{\cm}{{\cal M}}
\newcommand{\cn}{{\cal N}}
\newcommand{\co}{{\cal O}}
\newcommand{\cp}{{\cal P}}
\newcommand{\cs}{{\cal S}}
\newcommand{\ct}{{\cal T}}
\newcommand{\cx}{{\cal X}}
\newcommand{\cy}{{\cal Y}}
\newcommand{\cz}{{\cal Z}}

\thispagestyle{empty}

\vspace*{1cm}
\begin{center}
{\Large\bf Strict positivity for the principal eigenfunction of \\[5pt]
elliptic operators with various boundary conditions} \\[10mm]
\large W. Arendt$^1$, A.F.M. ter Elst$^2$ and J. Gl\"uck$^3$

\end{center}

\vspace{5mm}

\begin{center}
{\bf Abstract}
\end{center}

\begin{list}{}{\leftmargin=1.7cm \rightmargin=1.7cm \listparindent=10mm 
   \parsep=0pt}
\item
We consider elliptic operators with measurable coefficients and
Robin boundary conditions 
on a bounded domain $\Omega \subset \Ri^d$ and show that the first eigenfunction
$v$ satisfies $v(x) \ge \delta > 0$ for all $x \in \overline{\Omega}$, even 
if the boundary $\partial \Omega$ is only Lipschitz continuous.
Under such weak regularity assumptions 
the Hopf--Ole{\u{\i}}nik boundary lemma is not available;
instead we use a new approach based on an abstract positivity improving condition
for semigroups that map $L_p(\Omega)$ into $C(\overline{\Omega})$.
The same tool also yields corresponding results for Dirichlet or mixed boundary 
conditions.

Finally, we show that our results can be used to derive strong minimum and maximum
principles for parabolic and elliptic equations.

\end{list}

\vspace{1cm}
\noindent
July 2020

\vspace{5mm}

\vspace{5mm}

\vspace{1mm}

\noindent
{\bf Home institutions:}    \\[3mm]
\begin{tabular}{@{}cl@{\hspace{10mm}}cl}
1. & Institute of Applied Analysis  & 
  2. & Department of Mathematics   \\
& University of Ulm   & 
  & University of Auckland   \\
& Helmholtzstr.\ 18 & 
  & Private bag 92019  \\
& 89081 Ulm & 
  & Auckland  \\
& Germany  & 
  & New Zealand  \\
& email: wolfgang.arendt@uni-ulm.de &
  & email: terelst@math.auckland.ac.nz   \\[8mm]
\end{tabular}  \\
\begin{tabular}{@{}cl@{\hspace{10mm}}cl}
3. &Faculty of Computer Science and Mathematics  \\
&University of Passau  \\
&Innstra{\ss}e 33  \\
&94032 Passau  \\
&Germany  \\
& email: Jochen.Glueck@uni-passau.de
\end{tabular}

\newpage

\section{Introduction} \label{Smax1}

A frequent situation occurring in the study of elliptic but also parabolic boundary value 
problems with real coefficients on a bounded domain $\Omega \subset \Ri^d$
is the following.
The solutions satisfy a weak maximum principle and there exists a principal eigenvalue
with a principal eigenfunction $u_0$ satisfying $u_0(x) > 0$ a.e.\ on $\Omega$.
By elliptic regularity one also shows that $u_0 \in C(\overline{\Omega})$.
But what is not known is whether $u_0(x) \ge \delta > 0$ for all $x \in\overline{\Omega}$.
We shall show this under very weak regularity assumptions.
Such a result has applications for the construction of super- and subsolutions 
(see Daners--L\'opez-G\'omez \cite{DanersLG1}), but also for the asymptotic behaviour of 
parabolic problems.

Let us describe a concrete situation.
Let $\Omega \subset \Ri^d$ be a bounded Lipschitz domain and $\beta \in L_\infty(\partial \Omega)$.
Given $u_0 \in C(\overline{\Omega})$ there exists a unique 
$u \in C([0,\infty) \times \overline{\Omega}) \cap C^\infty((0,\infty) \times \Omega)$ 
such that 
\[
\begin{array}{r@{}c@{}l@{\quad} l}
	\tfrac{\partial}{\partial t} u & {} = {} & \Delta u, \\[5pt]
	u(0,x) & = & u_0(x) & \mbox{for all } x \in \overline{\Omega}, \\[5pt]
	(\partial_\nu u)(t,z) + \beta(z) \, u(t,z) & = & 0 &
        \mbox{for all } z \in \partial \Omega \mbox{ and } t > 0.
\end{array}
\]
We shall show in Theorem~\ref{tmax501}
that if $u_0(x) \ge 0$ and $u_0 \neq 0$, then $u(t,x) > 0$ for all 
$x \in \overline{\Omega}$ and $t > 0$. 
This implies in particular that the principal eigenfunction $v \in C(\overline{\Omega})$ 
of the Robin Laplacian is strictly positive; that is, there exists a $\delta > 0$
such that $v(x) \geq \delta$ for all $x \in \overline \Omega$.
If $\partial \Omega$ and the eigenfunction are smooth enough, 
this property is known and can then be deduced 
from Hopf's maximum principle \cite{Bony, Lions} (on the interior) and the 
Hopf--Ole{\u{\i}}nik boundary lemma (see for instance \cite{AN}).
We shall prove the result for Lipschitz domains and arbitrary elliptic operators in 
divergence form with bounded real measurable coefficients, without any assumptions on the 
smoothness of the eigenfunction.
This new result 
is important for applications to non-linear problems 
(see for example \cite{DanersLG1}).

Our arguments are best placed in a more abstract situation. 
Let $S$ be a $C_0$-semigroup on $L_2(\Omega)$
which is positive and holomorphic.
Then $S$ is irreducible (see below for the definition) if and only if 
$S$ is positivity improving in the sense that 
if $u \geq 0$ and $u \neq 0$, then for each $t > 0$ one has 
$(S_t u)(x) > 0$ for almost every $x \in \Omega$. 
Irreducibility on $L_2(\Omega)$
is very easy to prove by the use of the Beurling--Deny--Ouhabaz criterion
\cite[Theorem~2.10]{Ouh5} 
and implies for the principal eigenfunction $v$ that $v(x) > 0$ almost
everywhere. In contrast to this, irreducibility in $C(\overline{\Omega})$ is much 
stronger: it implies that $v(x) \ge \delta > 0$ for all $x \in \overline{\Omega}$ and
some $\delta > 0$. Our main argument in Section~\ref{Smax3} shows that irreducibility
in $L_2(\Omega)$ already implies irreducibility in $C(\overline{\Omega})$ if
$S_t L_2(\Omega) \subset C(\overline{\Omega})$ for all $t > 0$ and if 
$(S_t|_{C(\overline{\Omega})})_{t > 0}$ is a $C_0$-semigroup on $C(\overline{\Omega})$
(see Theorem~\ref{tmax306} and Corollary~\ref{cmax301}).

In Section~\ref{Smax4} we 
will apply this result not only to elliptic problems with Robin boundary conditions,
but also to mixed boundary conditions, where we impose Neumann boundary conditions 
on a relatively open subset $N$ of $\partial \Omega$ and where we prove that 
$(S_t u)(x) > 0$ for all $t > 0$ and $x \in \Omega \cap N$ whenever $u \ge 0$ and $u \neq 0$.

In Section~\ref{Smax5} we will also prove a strong
minimum principle for the heat equation.
Given a continuous function $\psi$ on the parabolic boundary $\partial^* \Omega_T$
of the cylinder $\Omega_T = (0,T) \times \Omega$, there is a unique solution 
$u \in C(\overline{\Omega_T})$ of the heat equation $u_t = \Delta u$ which coincides
with $\psi$ on $\partial^* \Omega_T$.
We shall show that if $\psi \ge 0$ and $u(t_0,x_0) = 0$ for some $(t_0,x_0) \in \Omega_T$, 
then  $u(t,x) = 0$ for all $(t,x) \in \Omega_T$ such that $t < t_0$.
Again, this also remains true if the Laplacian $\Delta$ is replaced with
an elliptic operator (see Theorem~\ref{tmax504}). As a nice consequence, we also obtain 
a new proof of the strong parabolic maximum principle for elliptic operators
in divergence form with bounded real measurable coefficients.

The paper is organised as follows.
After a general introduction to irreducibility in Section~\ref{Smax2}, we establish 
our main abstract result in Section~\ref{Smax3}.
Principal eigenvectors for elliptic problems with diverse boundary conditions are 
considered in Section~\ref{Smax4} and the strong
minimum principle
is established in 
Section~\ref{Smax5}.
For the parabolic operator we have two notions of solutions: mild and weak. 
For the mild solution we do not need any regularity on the coefficients 
of the operator, but in order to define the weak solutions we need some 
differentiability. 
Under these differentiability conditions we show that weak solutions and 
mild solutions are equivalent.
For the latter equivalence we need a regularity result, for which we 
provide an elementary proof in the appendix.

\section{Preliminaries: Irreducibility} \label{Smax2}

In this section we recall the notions 
of positivity and irreducibility as well as
some results which are used later.
General references for this topic are \cite{Nag} and \cite{BatkaiKR}.

Throughout the section, let $E$ be a Banach lattice over $\Ki$, where the field 
$\Ki$ is either $\Ri$ or $\Ci$. We are especially interested 
in the following cases.

\begin{exam} \label{xmax201}
Let $\Omega \subset \Ri^d$ be an open non-empty set.
The following spaces are examples of Banach lattices.
\begin{tabel}
\item \label{xmax201-1}
$E = L_p(\Omega)$, where $p \in [1,\infty)$.
\item \label{xmax201-2}
$E = C(\overline \Omega)$, if $\Omega$ is bounded.
\item \label{xmax201-3}
$E = C_0(\Omega)$, the closure in $L_\infty(\Omega)$
of the space $C_c(\Omega)$ of all continuous functions with 
compact support.
\end{tabel}
\end{exam}

Let $E_+ = \{ u \in E : u \geq 0 \} $ be the {\bf positive cone} of $E$.
An {\bf ideal} of $E$ is a subspace $J$ of $E$ such that 
\begin{tabel}
\item 
if $u \in J$, then $|u| \in J$ and 
\item
if $u \in J$, $v \in E$ and $0 \leq v \leq u$, then $v \in J$.
\end{tabel}
The closure of an ideal is again an ideal.
The closed ideals can be characterised in the case of Example~\ref{xmax201}.
Let $\Omega \subset \Ri^d$ be an open set.
If $p \in [1,\infty)$ and 
$E = L_p(\Omega)$, then $J \subset E$ is a closed ideal if and only if 
there exists a measurable subset $B$ of $\Omega$ such that 
$J = \{ u \in E : u|_B = 0 \mbox{ almost everywhere} \} $
(see \cite[Section~III.1, Example~1]{Schae2}).
If $\Omega$ is bounded, and $E = C(\overline \Omega)$, then 
$J \subset E$ is a closed ideal if and only if 
there exists a closed set $B \subset \overline \Omega$ 
such that $J = \{ u \in C(\overline \Omega) : u|_B = 0 \} $ 
(see \cite[Section~III.1, Example~2]{Schae2}).
Finally, if $E = C_0(\Omega)$, then 
$J \subset E$ is a closed ideal if and only if 
there exists a closed set $B \subset \Omega$ 
such that $J = \{ u \in C_0(\Omega) : u|_B = 0 \} $ 
(see \cite[Proposition~10.14]{BatkaiKR}).

Note that $u \geq 0$ in $L_p(\Omega)$ means that $u(x) \in [0,\infty)$
for {\em almost} every $x \in \Omega$, whilst $u \geq 0$ in $C(\overline \Omega)$
means that $u(x) \in [0,\infty)$ for {\em all} $x \in \Omega$.
We write $u > 0$ if $u \geq 0$ and $u \neq 0$.
Note that $u \neq 0$ in $L_p(\Omega)$ means that 
$ \{ x \in \Omega : u(x) \neq 0 \} $ is not a null set.

If $u \geq 0$, then we denote by 
\[
E_u = \{ v \in E : \mbox{there exists an $n \in \Ni$ such that } |v| \leq n \, u \}
\]
the {\bf principal ideal} generated by $u$. 
It is easy to verify that this is indeed an ideal.
We write $u \gg 0$ if $\overline{E_u} = E$. 
In the literature of Banach lattices, such an element $u$ is called a 
{\bf quasi-interior point}.
As a remark, quasi-interior points can be characterized by an approximation 
condition.
Schaefer \cite[Theorem II.6.3]{Schae2} proved that 
a vector $u \in E_+$ is a quasi-interior point if and only if 
$\lim_{n \to \infty} v \land nu = v$ for every $v \in E_+$.

If $E = L_p(\Omega)$, then $u \gg 0$ if and only if 
$u(x) > 0$ for almost every $x \in \Omega$.
If $\Omega$ is bounded and $E = C(\overline \Omega)$, 
then $u \gg 0$ if and only if $u(x) > 0$ for all $x \in \overline \Omega$. 
So by compactness, $u \gg 0$ if and only if there exists a $\delta > 0$
such that $u(x) \geq \delta$ for all $x \in \overline \Omega$, 
which is the case if and only if $u$ is an interior point of the 
positive cone $E_+$.
If $E = C_0(\Omega)$, then $u \gg 0$ if and only if 
$u(x) > 0$ for all $x \in \Omega$. 
Note that the interior of $E_+$ is empty if $E = C_0(\Omega)$ or 
$E = L_p(\Omega)$.

Let also $F$  be a Banach lattice.
A linear map $R \colon E \to F$ is called {\bf positive} 
if $R E_+ \subset F_+$.
Positivity implies that $R$ is continuous by \cite[Theorem~II.5.3]{Schae2}.
We write $R \geq 0$ to express that $R$ is positive.
The set of all positive linear functionals on $E$ is denoted by~$E'_{\,+}$.
If $E = L_p(\Omega)$, then $E'_{\,+} = L_{p'}(\Omega)_+$, where 
$p' \in (1,\infty]$ is the dual exponent.
If $\Omega$ is bounded and 
$E = C(\overline \Omega)$, then $E'_{\,+}$ is isomorphic to 
all finite (positive) Borel measures on $\overline \Omega$.
If $E = C_0(\Omega)$, then $E'_{\,+}$ is isomorphic to 
all (positive) finite Borel measures on $\Omega$.
For a proof of the last two statements, see \cite[Theorem~14.1]{HR1}.

An operator $R \colon E \to F$ is called {\bf positivity improving}
if $R u \gg 0$ in $F$ for all $u \in E$ with $u > 0$.
Positivity improving operators will be of central interest in this paper.

By a {\bf semigroup} on $E$ we mean a map $S \colon (0,\infty) \to \cl(E)$
such that $S_{t+s} = S_t \, S_s$ for all $s,t \in (0,\infty)$,
where we write $S_t = S(t)$.
We say that $S$ is a {\bf $C_0$-semigroup} if in addition 
$\lim_{t \downarrow 0} S_t u = u$ for all $u \in E$.
A semigroup $S$ is called {\bf positive} if $S_t \geq 0$ for all $t > 0$
and $S$ is called {\bf positivity improving} if $S_t$ is positivity improving
for all $t > 0$.
A semigroup $S$ is called {\bf irreducible} if it does not leave invariant
any nontrivial closed ideal; that is, 
if $J \subset E$ is a closed ideal and $S_t J \subset J$ for all $t > 0$,
then $J = E$ or $J = \{ 0 \} $.

Irreducibility is independent of $p$ for compatible semigroups.

\begin{lemma} \label{lmax202}
Let $\Omega \subset \Ri^d$ be open and $p_1,p_2 \in [1,\infty)$.
Let $S^{(1)}$ and $S^{(2)}$ be semigroups on $L_{p_1}(\Omega)$ 
and $L_{p_2}(\Omega)$.
Suppose that $S^{(1)}$ and $S^{(2)}$ are compatible, that is we have
$S^{(1)}_t u = S^{(2)}_t u$ almost everywhere for all $t > 0$ and 
$u \in L_{p_1}(\Omega) \cap L_{p_2}(\Omega)$.
Then $S^{(1)}$ is irreducible if and only if $S^{(2)}$ is irreducible.
\end{lemma}
\begin{proof}
Suppose that $S^{(1)}$ is irreducible.
Let $B \subset \Omega$ be measurable and set 
$J_2 = \{ u \in L_{p_2}(\Omega) : u|_B = 0 \} $.
Furthermore, suppose that $S^{(2)}_t J_2 \subset J_2$ for all $t > 0$.
Let $J_1 = \{ u \in L_{p_1}(\Omega) : u|_B = 0 \} $.
Let $t > 0$ and $u \in J_1$.
For all $n \in \Ni$ let 
$V_n = \{ x \in \Omega : \|x\| \leq n \mbox{ and } |u(x)| \leq n \} $.
Then $u \, \one_{V_n} \in J_2 \cap L_{p_1}(\Omega)$.
So $S^{(1)}_t (u \, \one_{V_n}) = S^{(2)}_t (u \, \one_{V_n}) \in J_2$.
Therefore $(S^{(1)}_t (u \, \one_{V_n}))|_B = 0$ and 
$S^{(1)}_t (u \, \one_{V_n}) \in J_1$.
Then $S^{(1)}_t u = \lim_{n \to \infty} S^{(1)}_t (u \, \one_{V_n}) \in J_1$
since $J_1$ is closed.
So $S^{(1)}_t J_1 \subset J_1$.
Because $S^{(1)}$ is irreducible, one concludes that $|B| = 0$ or 
$|\Omega \setminus B| = 0$ and $S^{(2)}$ is irreducible.
\end{proof}

In general, a positive and irreducible $C_0$-semigroup does not need 
to be positivity 
improving.
An counterexample is the rotation semigroup on $L_2(\Ti)$,
where $\Ti$ is the unit circle in~$\Ci$. 
The situation changes, however, if the semigroup is also holomorphic.

\begin{thm} \label{tmax203}
Let $S$ be a positive irreducible holomorphic $C_0$-semigroup on $E$.
Then $S$ is positivity improving.
\end{thm}
\begin{proof}
See Majewski--Robinson \cite[Theorem~3]{MajewskiRobinson}.
\end{proof}

In the following proposition we collect a number of known spectral
theoretic properties of positive semigroups.

\begin{prop} \label{pmax204}
Let $S$ be a positive irreducible $C_0$-semigroup in $E$ 
and suppose that its generator $-A$ has compact resolvent.
Then one has the following.
\begin{tabel}
\item \label{pmax204-1}
$\sigma(A) \neq \emptyset$.
\item \label{pmax204-2}
The number $\lambda_1 := \inf \{ \RRe \lambda : \lambda \in \sigma(A) \}$ 
is an eigenvalue of $A$ (and consequently, the infimum is actually a minimum).
\item \label{pmax204-3}
There exists a $u \in D(A)$ such that $A u = \lambda_1 \, u$ and $u \gg 0$.
\item \label{pmax204-4}
The algebraic multiplicity of the eigenvalue $\lambda_1$ is one.
\end{tabel}
\end{prop}
\begin{proof}
`\ref{pmax204-1}'.
We may assume that $\dim E \ge 2$. 
Then by a result of de Pagter \cite[Theorem~3]{Pag2} 
every compact, positive and irreducible operator on $E$ has non-zero spectral radius. 
If we apply this  
to the resolvent of $A$, the assertion follows.
`\ref{pmax204-2}'.
See \cite[Corollary~12.9]{BatkaiKR}.
`\ref{pmax204-3}'.
It follows from the Krein--Rutman theorem, see for example
\cite[Theorem~12.15]{BatkaiKR}, that there exists a $u \in D(A)$
with $A u = \lambda_1 \, u$ and $u > 0$.
Then the statement follows from \cite[Proposition~14.12(a)]{BatkaiKR}.
`\ref{pmax204-4}'.
This follows from \cite[Proposition~C-III.3.5]{Nag}.
\end{proof}

Note that since $A$ has compact resolvent, $\lambda_1$ is an isolated 
point of the spectrum.
Therefore Proposition~\ref{pmax204}\ref{pmax204-4} means that 
the spectral projection for $\lambda_1$
has rank one.

If $S$ is a positive irreducible  $C_0$-semigroup 
whose generator $-A$ has compact resolvent, then we call 
$\min \{ \RRe \lambda : \lambda \in \sigma(A) \} $ the 
{\bf principal eigenvalue} of $A$.
It follows from Proposition~\ref{pmax204}
that the principal 
eigenvalue has a unique eigenvector $u$ such that $u \geq 0$ and $\|u\| = 1$.
We call $u$ the {\bf principal eigenvector} of $A$.
One has $u \gg 0$.

\section{Irreducibility on $C(\overline \Omega)$ and $C_0(\Omega)$} \label{Smax3}

In this section we consider a positive irreducible holomorphic $C_0$-semigroup
on $L_p(\Omega)$ which maps $L_p(\Omega)$ into $C(\overline \Omega)$
or $C_0(\Omega)$.
Under a mild additional condition we shall prove that the semigroup 
obtained by restriction to $C(\overline \Omega)$ or $C_0(\Omega)$
is again irreducible.

In Subsection~\ref{Smax3.1} we prove a not too difficult but very powerful abstract result
that is the basis of everything that follows. In Subsections~\ref{Smax3.2} and~\ref{Smax3.3}
we consider the special cases $C(\overline \Omega)$ and
$C_0(\Omega)$, respectively. We close the section with a brief remark on the long-term behaviour of 
positive semigroups in Subsection~\ref{Smax3.4}.

\subsection{An abstract positivity improvement result} \label{Smax3.1}

We start with a general theorem about positivity in a single point. It is the main ingredient for
our proofs of strict positivity in Section~\ref{Smax4}.
Let $\Omega \subset \Ri^d$ be an open non-empty set and $X$ a set 
such that $\Omega \subset X \subset \overline{\Omega}$.
If $p \in [1,\infty)$ and $u \in L_p(\Omega)$, then we say that 
$u \in C(X)$ if there exists a (necessarily unique) $\tilde u \in C(X)$
such that $\tilde u|_\Omega = u$ almost everywhere on $\Omega$.
Note that $\partial \Omega$ might have positive Lebesgue measure.
In the sequel we will identify $u$ and $\tilde u$.
For example, in the next theorem we identify 
$S_t u$ and $(S_t u)\,\widetilde{\;}\,$.

\begin{thm} \label{tmax306}
Let $\Omega \subset \Ri^d$ be an open non-empty set and $p \in [1,\infty)$.
Let $S$ be a positive irreducible holomorphic $C_0$-semigroup on $L_p(\Omega)$.
Next let $X$ be a set such that 
$\Omega \subset X \subset \overline{\Omega}$.
Finally let $x \in X$.
Suppose 
\begin{tabelR}
\item \label{tmax306-c1}
$S_t L_p(\Omega) \subset C(X)$ for all $t > 0$
and 
\item \label{tmax306-c2}
there are $t > 0$ and $w \in L_p(\Omega)$
such that $(S_t w)(x) \neq 0$.
\end{tabelR}
Then  $(S_t u)(x) > 0$ for all $t > 0$ 
and $u \in L_p(\Omega)$ with $u \geq 0$ and $u \neq 0$.
\end{thm}

In what follows, typical choices for $X$ are $X = \Omega$ or $X = \overline \Omega$.
We also have, however, an application in Theorem~\ref{tmax422} 
for elliptic operators with mixed boundary
conditions, where $X$ is chosen strictly between $\Omega$ and $\overline{\Omega}$. Let us also remark that,
while Theorem~\ref{tmax306} works pointwise, we are in fact most interested in the case
where condition~\ref{tmax306-c2}, and then also the conclusion of the theorem, is valid for all $x \in X$ 
instead of merely a single point.

\begin{proof}[\bf Proof of Theorem~\ref{tmax306}]
The map $u \mapsto (S_t u)(x)$ from $L_p(\Omega)$ into $\Ci$ is positive, 
hence it is continuous by \cite[Theorem~II.5.3]{Schae2}.
By assumption~\ref{tmax306-c2} there exist $s > 0$ and $w \in L_p(\Omega)$ such that 
$(S_s w)(x) \neq 0$.
Without loss of generality we may assume that $w \geq 0$. Therefore $(S_s w)(x) > 0$.

Let $t \in (0,\infty)$ and $u \in L_p(\Omega)$ with $u > 0$.
There are $t_1,t_2 \in (0,\infty)$ such that $t = t_1 + t_2$ and $t_1 < s$.
According to Theorem~\ref{tmax203} we have $S_{t_2} u \gg 0$ in $L_p(\Omega)$, so 
it follows from the Lebesgue dominated convergence theorem that
\[
	\lim_{n \to \infty} S_{s-t_1} w \wedge n \, S_{t_2} u = S_{s-t_1} w
\]
with respect to the norm in $L_p(\Omega)$.
By the continuity that we established in the beginning, it follows that
\[
\lim_{n \to \infty} \Big( S_{t_1} (S_{s-t_1} w \wedge n \, S_{t_2} u) \Big)(x)
= (S_s w)(x) > 0.
\]
Consequently there exists an $n \in \Ni$ such that 
$\Big( S_{t_1} (S_{s-t_1} w \wedge n \, S_{t_2} u) \Big)(x) > 0$.
But then 
\[
0 < \Big( S_{t_1} (S_{s-t_1} w \wedge n \, S_{t_2} u) \Big)(x)
\leq \Big( S_{t_1} (n \, S_{t_2} u ) \Big)(x)
= n \, (S_{t_1 + t_2} u)(x)
= n \, (S_t u)(x)
\]
and the theorem follows.
\end{proof}

For the convenience of the reader, as well as for the sake of later reference,
we explicitly state a few consequences of Theorem~\ref{tmax306} 
in the following subsections.

\subsection{Irreducibility on $C(\overline{\Omega})$} \label{Smax3.2}

As a special case of Theorem~\ref{tmax306} one obtains the following 
result for $X = \overline \Omega$ if $\Omega$ is bounded.

\begin{cor} \label{cmax301}
Let $\Omega \subset \Ri^d$ be a bounded open set and $p \in [1,\infty)$.
Let $S$ be a positive irreducible holomorphic $C_0$-semigroup on $L_p(\Omega)$.
Suppose 
\begin{tabelR}
\item \label{cmax301-c1}
$S_t L_p(\Omega) \subset C(\overline \Omega)$ for all $t > 0$
and 
\item \label{cmax301-c2}
for all $x \in \overline \Omega$ there are $t > 0$ and $w \in L_p(\Omega)$
such that $(S_t w)(x) \neq 0$.
\end{tabelR}
Then the following holds.
\begin{tabel}
\item \label{cmax301-1}
For all $t > 0$ the operator $S_t \colon L_p(\Omega) \to C(\overline \Omega)$
is positivity improving.
This means $(S_t u)(x) > 0$ for all $x \in \overline \Omega$, $t > 0$ 
and $u \in L_p(\Omega)$ with $u \geq 0$ and $u \neq 0$.
\item \label{cmax301-2}
For all $t > 0$ define $T_t = S_t|_{C(\overline \Omega)} \colon C(\overline \Omega) \to C(\overline \Omega)$.
Then the semigroup $T$ 
is irreducible on $C(\overline \Omega)$.
\end{tabel}
\end{cor}
\begin{proof}
`\ref{cmax301-1}'.
This is a special case of Theorem~\ref{tmax306}.

`\ref{cmax301-2}'.
This follows immediately from the charactisation of closed ideals in 
$C(\overline \Omega)$ and Statement~\ref{cmax301-1}.
\end{proof}

Note that Condition~\ref{cmax301-c2} in Corollary~\ref{cmax301} is satisfied 
if Condition~\ref{cmax301-c1} is valid and $T$ is a $C_0$-semigroup
on $C(\overline \Omega)$, where $T$ is as in Statement~\ref{cmax301-2}.
It is also satisfied if there exists a $t > 0$ such that 
$S_t \one_\Omega = \one_{\overline \Omega}$.

It is a consequence of Corollary~\ref{cmax301} that the semigroup has a 
strictly positive kernel if $\Omega$ is bounded.

\begin{cor} \label{cmax304}
Let $\Omega \subset \Ri^d$ be a bounded open set and 
$S$ a semigroup on $L_2(\Omega)$.
Let $p \in [2,\infty)$.
Suppose that 
\begin{tabelR}
\item \label{cmax304-1}
$S_t L_p(\Omega) \subset C(\overline \Omega)$
and $S_t^* L_p(\Omega) \subset C(\overline \Omega)$ for all $t > 0$,
\item \label{cmax304-2}
for all $x \in \overline \Omega$ there are $t > 0$ and $w \in L_p(\Omega)$
such that $(S_t w)(x) \neq 0$, and 
\item \label{cmax304-3}
for all $x \in \overline \Omega$ there are $t > 0$ and $w \in L_p(\Omega)$
such that $(S^*_t w)(x) \neq 0$.
\end{tabelR}
Further suppose that $(S_t|_{L_p(\Omega)})_{t > 0}$ 
is a positive irreducible holomorphic $C_0$-semigroup on $L_p(\Omega)$.
Then for $t > 0$ there exists a function 
$K_t \in C(\overline \Omega \times \overline \Omega)$ such that 
\[
(S_t u)(x) = \int_\Omega K_t(x,y) \, u(y) \, dy
\]
for all $u \in L_2(\Omega)$ and $x \in \overline \Omega$.
Moreover, 
$K_t(x,y) > 0$ for all $x,y \in \overline \Omega$ and $t > 0$.
\end{cor}
\begin{proof}
We would like to apply \cite[Theorem 2.1]{AE8}.
To do so, we need a semigroup on an $L_2$-space over $\overline{\Omega}$,
which needs a bit of care since the boundary
$\partial \Omega$ might have non-zero Lebesgue measure.
Let $\lambda$ denote the Lebesgue measure on $\Omega$ 
and define the Borel measure $\mu$ on 
$\overline \Omega$ given by 
\[
\mu(B) = \lambda(B \cap \Omega)
\]
for each Borel set $B \subset \overline{\Omega}$.
Then $\mu$ is strictly positive on each
non-empty open subset of~$\overline{\Omega}$.
Moreover, for each 
$q \in [1,\infty)$, the embedding $L_q(\Omega) \hookrightarrow L_q(\overline{\Omega},\mu)$,
given by extending functions on $\Omega$ by $0$ on $\partial \Omega$,
is an isomorphism.
Hence we can transport the semigroup $S$ on $L_2(\Omega)$
to a semigroup $T$ on $L_2(\overline{\Omega},\mu)$.
Then Assumption~\ref{cmax304-1} implies that 
$T_t L_p(\overline \Omega,\mu) \subset C(\overline \Omega)$
and $T_t^* L_p(\overline \Omega,\mu) \subset C(\overline \Omega)$ for all $t > 0$.
It follows from \cite[Theorem~2.1]{AE8} that the operator $T_t$ has a continuous 
kernel $K_t \in C(\overline \Omega \times \overline \Omega)$ for all $t > 0$.
Then
\[
(S_t u)(x) = \int_\Omega K_t(x,y) \, u(y) \, dy
\]
for all $t > 0$, $u \in L_2(\Omega)$
and $x \in \overline \Omega$.
If $t > 0$, then $K_t \geq 0$ almost everywhere on $\Omega \times \Omega$ 
since  $S_t$ is a positive operator.
Hence by continuity $K_t(x,y) \geq 0$ for all $t > 0$ and $x,y \in \overline \Omega$.

Finally, let $t > 0$ and $x,y \in \overline \Omega$.
By Assumption~\ref{cmax304-3} there exist $s > 0$ and $w \in L_p(\Omega)$ such that 
$(S^*_s w)(y) \neq 0$.
There are $t_1,t_2 \in (0,\infty)$ such that $t = t_1 + t_2$ and $t_1 < s$.
Define $v \colon \overline \Omega \to \Ri$ by $v(z) = K_{t_1}(z,y)$.
Then $v \neq 0$ since 
\[
0 \neq (S_s^*w)(y) = (S_{t_1}^* S_{s-t_1}^* w)(y) = \int_\Omega v(z) \, (S_{s-t_1}^*w)(z) \, dz.
\]
Therefore, $K_t(x,y) = (S_{t_2} v)(x) > 0$ by Corollary~\ref{cmax301}\ref{cmax301-1},
where we use Assumptions~\ref{cmax304-2} and~\ref{cmax304-1}.
\end{proof}

\subsection{Irreducibility on $C_0(\Omega)$} \label{Smax3.3}

Analogously to Corollary~\ref{cmax301} one can use Theorem~\ref{tmax306} 
to derive irreducibility for semigroups on $C_0(\Omega)$.
This yields the following corollary.
Note that $\Omega$ does not have to be bounded in this subsection.

\begin{cor} \label{cmax302}
Let $\Omega \subset \Ri^d$ be an open set and $p \in [1,\infty)$.
Let $S$ be a positive irreducible holomorphic $C_0$-semigroup on $L_p(\Omega)$.
Suppose 
\begin{tabelR}
\item \label{cmax302-c1}
$S_t L_p(\Omega) \subset C_0(\Omega)$ for all $t > 0$
and 
\item \label{cmax302-c2}
for all $x \in \Omega$ there are $t > 0$ and $w \in L_p(\Omega)$
such that $(S_t w)(x) \neq 0$.
\end{tabelR}
Then one has the following.
\begin{tabel}
\item \label{cmax302-1}
For all $t > 0$ the operator $S_t \colon L_p(\Omega) \to C_0(\Omega)$
is positivity improving.
This means $(S_t u)(x) > 0$ for all $x \in \Omega$, $t > 0$ 
and $u \in L_p(\Omega)$ with $u \geq 0$ and $u \neq 0$.
\item \label{cmax302-2}
Suppose that for all $t > 0$ the operator $S_t|_{C_0(\Omega) \cap L_p(\Omega)}$
extends to a continuous operator $T_t$ from $C_0(\Omega)$ into $C_0(\Omega)$.
Then the semigroup $T$ is irreducible on $C_0(\Omega)$.
\end{tabel}
\end{cor}

Note that if $\Omega$ is bounded, then $C_0(\Omega) \subset L_p(\Omega)$ 
and the operator $S_t|_{C_0(\Omega)}$ is indeed a 
continuous operator from $C_0(\Omega)$ into $C_0(\Omega)$.
Moreover Condition~\ref{cmax302-c2} in Corollary~\ref{cmax302} is satisfied 
if $\Omega$ is bounded, Condition~\ref{cmax302-c1} is valid 
and $T$ is a $C_0$-semigroup
on $C_0(\Omega)$, where $T$ is defined as in~\ref{cmax302-2}.

Similarly as in the proof of Corollary~\ref{cmax304} 
we obtain a kernel for the semigroup in case $\Omega$ is bounded.

\begin{cor} \label{cmax305}
Let $\Omega \subset \Ri^d$ be a bounded open set and 
$S$ a semigroup on $L_2(\Omega)$.
Let $p \in [2,\infty)$.
Suppose that 
\begin{tabelR}
\item \label{cmax305-1}
$S_t L_p(\Omega) \subset C_0(\Omega)$
and $S_t^* L_p(\Omega) \subset C_0(\Omega)$ for all $t > 0$,
\item \label{cmax305-2}
for all $x \in \Omega$ there are $t > 0$ and $w \in L_p(\Omega)$
such that $(S_t w)(x) \neq 0$, and 
\item \label{cmax305-3}
for all $x \in \Omega$ there are $t > 0$ and $w \in L_p(\Omega)$
such that $(S^*_t w)(x) \neq 0$.
\end{tabelR}
Further suppose that $(S_t|_{L_p(\Omega)})_{t > 0}$ 
is a positive irreducible holomorphic $C_0$-semigroup on $L_p(\Omega)$.
Then for $t > 0$ there exists a function $K_t \in C_0(\Omega \times \Omega)$
such that 
\[
(S_t u)(x) = \int_\Omega K_t(x,y) \, u(y) \, dy
\]
for all $u \in L_2(\Omega)$ and $x \in \Omega$.
Moreover, 
$K_t(x,y) > 0$ for all $x,y \in \Omega$ and $t > 0$.
\end{cor}
\begin{proof}
All is similar as in the proof of Corollary~\ref{cmax304},
but one obtains that $K_t \in C(\overline \Omega \times \overline \Omega)$.
It remains to show that $K_t \in C_0(\Omega \times \Omega)$.
Let $t > 0$ and $x \in \partial \Omega$.
Since $S_t L_p(\Omega) \subset C_0(\Omega)$ it follows that 
\[
0 = (S_t u)(x) = \int_\Omega K_t(x,z) \, u(z) \, dz
\]
for all $u \in C_c(\Omega)$.
Hence $K_t(x,z) = 0$ for almost all $z \in \Omega$ and by continuity 
for all $z \in \Omega$.
By duality $K_t(z,y) = 0$ for all $z \in \Omega$ and $y \in \partial \Omega$.
Because 
\[
K_{2t}(x,y) = \int_\Omega K_t(x,z) \, K_t(z,y) \, dz,
\]
one deduces that 
$K_t(x,y) = 0$ if $x \in \partial \Omega$ or $y \in \partial \Omega$.
So $K_{2t} \in C_0(\Omega \times \Omega)$.
\end{proof}

\subsection{A note on the long-time behavior} \label{Smax3.4}

It is worthwhile to say a few sentences on how the properties that we discussed above
are related to the long-time of the semigroup.

\begin{remark} \label{rmax303}
In the situation of Corollary~\ref{cmax301}, 
and for bounded $\Omega$ in the situation of Corollary~\ref{cmax302},
the semigroup on $L_p(\Omega)$ consists of compact operators. 
Let $-A$ be the generator and $\lambda_1$ be the principal eigenvalue of $A$.
Then by \cite[Corollary~C-III.3.16]{Nag}
there is a {\bf spectral gap} in the sense that there exists an $\varepsilon > 0$ 
such that
$ \{ \lambda \in \sigma(A) : \RRe \lambda \leq \lambda_1 + \varepsilon \} = \{ \lambda_1 \} $.
Moreover, if $\lambda_1 = 0$, then $S_t$ converges in $\cl(L_p(\Omega))$
to a rank-one projection 
if $t \to \infty$ (see \cite[Proposition~C-III.3.5]{Nag}).
The same is valid in the case of Subsection~\ref{Smax3.3}
 for the semigroup $(S_t|_{C_0(\Omega)})_{t > 0}$
in $\cl(C_0(\Omega))$ by the semigroup property.
\end{remark}

\section{Strict positivity of principal eigenvectors and other applications} \label{Smax4}

In this section we use the theorems from Section~\ref{Smax3}
to establish strict positivity of the principal eigenfunction
of an elliptic operator for three types of boundary conditions:
Dirichlet (Subsection~\ref{Smax5.2}), Robin (Subsection~\ref{Smax5.1}) 
and mixed  (Subsection~\ref{Smax4.3}). For each of these boundary conditions we prove, besides strict positivity of the principal eigenvector, also irreducibility of the corresponding semigroup on a suitable space of continuous functions and a positivity improving property for the corresponding elliptic problem. Moreover, in Subsection~\ref{Smax5.4} we shall show that our results have a surprising consequence for elliptic problems with complex Robin boundary conditions.

Throughout this section, let $\Omega \subset \Ri^d$ be a bounded non-empty, open and connected 
set with boundary $\Gamma = \partial \Omega$.
For all $k,l \in \{1,\ldots,d\}$ let $a_{kl},b_k,c_k,c_0 \in L_\infty(\Omega,\Ri)$. We assume that the coefficients $a_{kl}$ satisfy a uniform ellipticity condition, namely that there exists a $\mu > 0$ such that, for almost all $x \in \Omega$, the inequality
\[
\RRe \sum_{k,l=1}^d a_{kl}(x) \, \xi_k \, \overline{\xi_l}
\geq \mu \, |\xi|^2
\]
holds for all $\xi \in \Ci^d$. In the following subsections we will define elliptic operators with the coefficients $a_{kl},b_k,c_k,c_0$ by means of form methods. 
Loosely speaking, the operator is equal to 
\[
u \mapsto 
- \sum_{k,l=1}^d \partial_l \, a_{kl} \, \partial_k u
- \sum_{k=1}^d \partial_k \, b_k \, u
+ \sum_{k=1}^d c_k \, \partial_k u
+ c_0 \, u
\]
with boundary conditions.
Moreover, depending on the boundary conditions, we will impose different regularity assumptions on the boundary of $\Omega$ in each subsection.

Most results in this section are a combination of theorems from the literature with 
Theorem~\ref{tmax306} and its corollaries. 
For each type of boundary conditions we state a theorem which describes a 
positivity improving property of the parabolic equation, 
and a corollary which yields a similar result for the corresponding elliptic equation.

\subsection{Dirichlet boundary conditions} \label{Smax5.2}

In this subsection we assume that $\Omega$ is Wiener regular. This means that for all $\varphi \in C(\Gamma)$ there exists a 
function $u \in C(\overline \Omega) \cap C^2(\Omega)$ such that $\Delta u = 0$
on $\Omega$ and $u|_\Gamma = \varphi$. For instance, $\Omega$ is Wiener regular if it has Lipschitz boundary.

Define the form $\gota \colon H^1_0(\Omega) \times H^1_0(\Omega) \to \Ci$ by 
\[
\gota(u,v)
= \int_\Omega \sum_{k,l=1}^d a_{kl} \, (\partial_k u) \, \overline{\partial_l v} 
   + \int_\Omega \sum_{k=1}^d b_k \, u \, \overline{\partial_k v} 
   + \int_\Omega \sum_{k=1}^d c_k \, (\partial_k u) \, \overline v
+ \int_\Omega c_0 \, u \, \overline v
.  \]
Then $\gota$ is a closed sectorial form.
Let $A$ be the m-sectorial operator on $L_2(\Omega)$ associated with $\gota$
and let $S$ be the semigroup generated by $-A$ on $L_2(\Omega)$.
Then $S$ is a positive semigroup by \cite[Theorem~2.6 or Corollary~4.3]{Ouh5}
and irreducibility of $S$ follows from \cite[Theorem~4.5]{Ouh5}.
Since the embedding $H^1_0(\Omega) \subset L_2(\Omega)$ is compact, the 
operator $A$ has compact resolvent.
If $t > 0$, then $S_t L_2(\Omega) \subset C_0(\Omega)$ by 
(8) in \cite{AE9}, where we used that $\Omega$ is Wiener regular.
For all $t > 0$ let 
$T_t = S_t|_{C_0(\Omega)} \colon C_0(\Omega) \to C_0(\Omega)$.
Then $T$ is a holomorphic
$C_0$-semigroup by \cite[Theorem~1.3]{AE9}.
Clearly $T$ is positive.
The following result shows that $T$ is also irreducible.

\begin{thm} \label{tmax502}
The operator $A$ on $L_2(\Omega)$ and the semigroup $T$ on $C_0(\Omega)$
have the following properties.
\begin{tabel}
\item \label{tmax502-1} For all $t > 0$ the operator $T_t$ is positivity improving.
In particular, the semigroup $T$ is irreducible.
\item \label{tmax502-2} Let $u$ be the principal eigenfunction of $A$. Then $u \in C_0(\Omega)$ and $u(x) > 0$ for all $x \in \Omega$.
\end{tabel}
\end{thm}

\begin{proof}
Statement~\ref{tmax502-1} follows from Corollary~\ref{cmax302} and Statement~\ref{tmax502-2}
from Proposition~\ref{pmax204}\ref{pmax204-3}.
\end{proof}

\begin{remark}
Theorem~\ref{tmax502} can also be derived from known, but much less elementary results 
from PDE.
Statement~\ref{tmax502-2} 
follows from the Harnack inequality (see for instance \cite[Theorem 8.20]{GT}).
Next, let $t > 0$ and let $K_t$ be the kernel of the operator $S_t$.
The De~Giorgi--Nash theorem implies that $K_t$ is continuous. 
Then the Harnack inequality shows that $K_t$ is strictly positive on $\Omega \times \Omega$.
Therefore $T_t$ is positivity improving and $T$ is irreducible.
\end{remark}

\begin{remark}
For the special case of the Laplacian, the strict parabolic maximum 
principle of Evans \cite[Section~2.3.3]{Evans} was used in \cite[Theorem~3.3]{Are4} to prove Theorem~\ref{tmax502}.
The proof in Evans, however, is based on a mean value property  which is 
not valid for operators with variable coefficients.
\end{remark}

We conclude this subsection with a positivity improving property for the corresponding elliptic problem. 
Since the semigroup $S$ has Gaussian kernel bounds it follows 
that the semigroup $S$  extends to a $C_0$-semigroup on $L_p(\Omega)$ for all $p \in [1,\infty)$.
We denote its generator by $-A_p$. As $A$ has compact resolvent, it follows that $A_p$ has compact resolvent too
 and that the spectrum of $A_p$ coincides with the spectrum of $A$ by \cite{Persson}.
We obtain from Theorem~\ref{tmax502} the following corollary about regularity of the corresponding elliptic problem.

\begin{cor} \label{cmax503}
Let $\lambda \in \Ri$ be smaller than the first eigenvalue of $A$ and let $p \in (d/2,\infty)$. If $u \in D(A_p)$ and $(-\lambda \, I + A_p)u > 0$, then $u \in C_0(\Omega)$ and $u(x) > 0$ for all $x \in \Omega$.
\end{cor}

\begin{proof}
Denote the generator of $T$ by $-A_c$ and choose
$\mu \in \Ri$ such that $\mu < \lambda$ and $\mu < \inf \{\RRe \nu: \, \nu \in \sigma(A_c)\}$. 
The semigroup $T$ is irreducible according to Theorem~\ref{tmax502}\ref{tmax502-1}.
Hence it follows from \cite[Definition~C-III.3.1]{Nag} that the 
resolvent $(-\mu \, I +A_c)^{-1}$ is positivity improving on $C_0(\Omega)$.
Note that the operator $(-\mu \, I + A_c)^{-1}$ coincides 
with the restriction of $(-\mu \, I +A_p)^{-1}$ to $C_0(\Omega)$.

One deduces from \cite[Corollary~2.10]{AE9}
that the range of the resolvents $(-\lambda+A_p)^{-1}$ and $(-\mu+A_p)^{-1}$ are contained in 
$C_0(\Omega)$, where we use that $p > d/2$. 
Set $f = (-\lambda \, I + A_p)u$.
Then the resolvent identity implies that 
\begin{eqnarray*}
u = (-\lambda \, I + A_p)^{-1}f 
& = & (\lambda-\mu)(-\mu \, I+A_p)^{-1} (-\lambda \, I+A_p)^{-1}f + (-\mu \, I+A_p)^{-1}f  \\
& \geq & (\lambda-\mu)(-\mu \, I+A_p)^{-1} (-\lambda \, I+A_p)^{-1}f  \\
& = & (\lambda-\mu)(-\mu \, I+A_c)^{-1} (-\lambda \, I+A_p)^{-1}f 
\gg 0,
\end{eqnarray*}
where $\gg$ is to be understood in $C_0(\Omega)$. 
This proves the corollary.
\end{proof}

Note that $\sigma(A_c) = \sigma(A_2)$ by \cite[Proposition~3.10.3]{ABHN}.

\subsection{Robin boundary conditions} \label{Smax5.1}

In this subsection we assume in addition that $\Omega$ has Lipschitz boundary. 
Further let $\beta \in L_\infty(\Gamma,\Ri)$.
Define the form $\gota \colon H^1(\Omega) \times H^1(\Omega) \to \Ci$ by 
\begin{eqnarray*}
\gota(u,v)
& = & \int_\Omega \sum_{k,l=1}^d a_{kl} \, (\partial_k u) \, \overline{\partial_l v} 
   + \int_\Omega \sum_{k=1}^d b_k \, u \, \overline{\partial_k \, v} 
   + \int_\Omega \sum_{k=1}^d c_k \, (\partial_k u) \, \overline v
\\* 
& & \hspace*{20mm} {}
+ \int_\Omega c_0 \, u \, \overline v
+ \int_\Gamma \beta \, (\Tr u) \, \overline{\Tr v}
.  
\end{eqnarray*}
Then $\gota$ is a closed sectorial form. Let $A$ be the m-sectorial operator 
on $L_2(\Omega)$ associated with $\gota$ and let $S$ be the semigroup generated by $-A$ 
on $L_2(\Omega)$. Then $S$ is a positive semigroup by \cite[Theorem~2.6]{Ouh5}.
Moreover, $S$ is irreducible on $L_2(\Omega)$ by 
\cite[Corollary~2.11]{Ouh5} together with the discussion on 
page~106 in \cite{Ouh5}.
Since $\Omega$ is bounded and Lipschitz, the operator~$A$ has compact resolvent.
If $t > 0$, then $S_t L_2(\Omega) \subset C(\overline \Omega)$ by 
\cite[Remark~6.2]{AE8}.
Let $T_t = S_t|_{C(\overline \Omega)} \colon C(\overline \Omega) \to C(\overline \Omega)$
for all $t > 0$.
Then $T$ is a $C_0$-semigroup by \cite[Remark~6.2]{AE8}.

\begin{thm} \label{tmax501}
The operator $A$ on $L_2(\Omega)$ and the semigroup $T$ on $C(\overline{\Omega})$
have the following properties.
\begin{tabel}
\item \label{tmax502check-1} For all $t > 0$ the operator $T_t$ is positivity improving. 
In particular, the semigroup $T$ is irreducible.
\item \label{tmax502check-2} Let $u$ be the principal eigenvalue of $A$. Then $u \in C(\overline{\Omega})$ and $u(x) > 0$ for all $x \in \overline{\Omega}$.
\end{tabel}
\end{thm}

\begin{proof}
Statement~\ref{tmax502check-1} follows from Corollary~\ref{cmax301} and 
Statement~\ref{tmax502check-2} from Proposition~\ref{pmax204}\ref{pmax204-3}.
\end{proof}

Note that it follows again from the Harnack inequality that $u(x) > 0$ for all $x \in \Omega$.
The above theorem, however, says much more, namely that $u$ is also strictly positive on the boundary
of $\Omega$ and hence, bounded away from $0$.  This is of interest in the 
study of nonlinear equations, and is new under such general conditions as we have here.
Under much stronger regularity conditions, for instance if $\Omega$ has a $C^2$-boundary and all coefficients are smooth, one can
of course deduce Theorem~\ref{tmax501} from Hopf's minimum principle,
see for example \cite[Theorem~1.2]{LopezGomez1}

Again, we also derive a corresponding elliptic result. 
By the Gaussian kernel bounds of \cite[Theorem~2.2]{Daners8} and \cite{Daners}
the semigroup $S$ on $L_2(\Omega)$ extrapolates to a $C_0$-semigroup
on $L_p(\Omega)$ for all $p \in [1,\infty)$, 
whose generator we denote by $-A_p$. If $p > d/2$, then the resolvent operators of
$A_p$ map $L_p(\Omega)$ into $C(\overline{\Omega})$
by \cite[Theorem~3.14(iv)]{Nit4}.
Hence by exactly the same arguments as in the proof of Corollary~\ref{cmax503} we 
can obtain the following consequence of Theorem~\ref{tmax501}.

\begin{cor} \label{cmax504}
Let $\lambda \in \Ri$ be smaller than the first eigenvalue of $A$, let $p \in (d/2,\infty)$ and $u \in D(A_p)$. Suppose that $(-\lambda \, I + A_p)u > 0$. Then $u \in C(\overline{\Omega})$ and $u(x) > 0$ for all $x \in \overline{\Omega}$.
\end{cor}

\subsection{The bottom of the spectrum for complex Robin boundary conditions} \label{Smax5.4}

In this subsection we consider complex Robin boundary conditions and show that 
Theorem~\ref{tmax501} has surprising
consequences for this situation.
Note that in Theorem~\ref{tmax501} the function $\beta$ is real valued.

As in Subsection~\ref{Smax5.1} we assume that $\Omega$ 
has Lipschitz boundary~$\Gamma$. 
For the coefficients of the differential operator 
we assume that $a_{kl} = a_{lk}$ and $b_k = c_k$ for all $k,l \in \{ 1,\ldots,d \}$.

For all $\beta \in L_\infty(\Gamma)$
define the form $\gota_\beta \colon H^1(\Omega) \times H^1(\Omega) \to \Ci$ by 
\begin{eqnarray*}
\gota_\beta(u,v)
& = & \int_\Omega \sum_{k,l=1}^d a_{kl} (\partial_k u) \, \overline{\partial_l v} 
   + \int_\Omega \sum_{k=1}^d b_k \, u \, \overline{\partial_k \, v} 
   + \int_\Omega \sum_{k=1}^d b_k \, (\partial_k u) \, \overline v
\\* 
& & \hspace*{20mm} {}
+ \int_\Omega c_0 \, u \, \overline v
+ \int_\Gamma \beta \, (\Tr u) \, \overline{\Tr v}
.  
\end{eqnarray*}
Then $\gota_\beta$ is a closed sectorial form.
Let $A_\beta$ be the m-sectorial operator associated with $\gota$.
Since $\Omega$ is bounded and Lipschitz, the operator $A_\beta$ has compact resolvent.
Note that $\gota_\beta$ is symmetric if $\beta$ is real valued.

\begin{prop} \label{pmax501.5}
Let $\beta \in L_\infty(\Gamma)$ with $\IIm \beta \neq 0$.
Then $\min \{ \RRe \lambda : \lambda \in \sigma(A_\beta) \} > \min \sigma(A_{\RRe \beta})$.
\end{prop}
\begin{proof}
Let $\lambda \in \sigma(A_\beta)$.
There exists a $u \in D(A_\beta)$ such that $A_\beta u = \lambda \, u$
and $\|u\|_{L_2(\Omega)} = 1$.
Then 
\[
\RRe \lambda
= \RRe (A_\beta u, u)_{L_2(\Omega)}
= \RRe \gota_\beta(u)
= \gota_{\RRe \beta}(u)
\geq \min \sigma(A_{\RRe \beta})
,  \]
where we used that $\gota_{\RRe \beta}$ is symmetric.
If $\RRe \lambda = \min \sigma(A_{\RRe \beta})$, then 
$\gota_{\RRe \beta}(u) = \min \sigma(A_{\RRe \beta})$.
So $u \in D(A_{\RRe \beta})$ and $A_{\RRe \beta} u = \lambda_1 \, u$, 
where $\lambda_1 = \min \sigma(A_{\RRe \beta})$
and we used Proposition~\ref{pmax204}\ref{pmax204-4}.
Using Theorem~\ref{tmax501}, one deduces that $u \in C(\overline \Omega)$
and $u(x) \neq 0$
for all $x \in \Gamma$ (even for all $x \in \overline \Omega$).
Let $\partial_\nu u$ denote the (weak) co-normal derivative of $u$.
Then $\partial_\nu u + \beta \, u|_\Gamma = 0$ in $L_2(\Gamma)$ since $u \in D(A_\beta)$.
But also $\partial_\nu u + (\RRe \beta) \, u|_\Gamma = 0$ in $L_2(\Gamma)$ 
since $u \in D(A_{\RRe \beta})$.
So $(\IIm \beta) \, u|_\Gamma = 0$ in $L_2(\Gamma)$ and hence $\IIm \beta = 0$ almost everywhere.
This is a contradiction.
\end{proof}

\subsection{Mixed boundary conditions} \label{Smax4.3}

In this subsection we assume that $\Omega$ has Lipschitz boundary.
Further, let $D \subset \partial \Omega$ be a closed set and 
define $N = \partial \Omega \setminus D$.
We consider elliptic differential operators with mixed 
boundary conditions where, roughly speaking, we wish to have Dirichlet boundary conditions on $D$ and 
Neumann boundary conditions on $N$. This yields an example 
where we apply Theorem~\ref{tmax306} with a set $X$ such that 
$\Omega \subsetneqq X \subsetneqq \overline \Omega$.

In contrast to the previous sections, we restrict ourselves to differential operators
with second order coefficients only, i.e.\ we assume that $b_k = c_k = c_0 = 0$ 
for all $k \in \{1,\ldots,d\}$.

Since the pure Dirichlet and pure Neumann case have been considered in the 
previous subsections, we assume that $D \neq \emptyset$ and $N \neq \emptyset$.
Let $\partial D$ be the boundary of $D$ in the relative topology of $\partial \Omega$.
We need a technical assumption which states that 
the set of points from the Dirichlet boundary part is large enough with respect to the boundary
measure (see \cite{ERe2}).
Precisely, we suppose that there exists a $\delta > 0$ such that 
for all $x \in \partial D$ and $r \in (0,1]$ there exists a $y \in D \cap B(x,r)$
such that 
\begin{equation}
N \cap B(y, \delta \, r) = \emptyset.
\label{etmax420;1}
\end{equation}
Next we introduce the generator.
Let 
\[
C^\infty_D(\Omega) 
= \{ \chi|_\Omega : \chi \in C_c^\infty(\Ri^d) \mbox{ and } D \cap \supp \chi = \emptyset \} 
\]
and let $W^{1,2}_D(\Omega)$ be the closure of $C^\infty_D(\Omega)$ in $W^{1,2}(\Omega)$.
Define the form $\gota \colon W^{1,2}_D(\Omega) \times W^{1,2}_D(\Omega) \to \Ci$ by 
\[
\gota(u,v)
= \int_\Omega \sum_{k,l=1}^d a_{kl} \, (\partial_k u) \, \overline{\partial_l v} 
.  \]
Then $\gota$ is a closed sectorial form.

Let $A$ be the operator associated with $\gota$ 
on $L_2(\Omega)$ and let $S$ be the semigroup generated by $-A$ on $L_2(\Omega)$.
Finally let $C_D(\overline \Omega) = \{ u \in C(\overline \Omega) : u|_D = 0 \} $.
We shall first show that $S$ leaves the space $C_D(\overline \Omega)$ invariant and that 
the restriction to $C_D(\overline \Omega)$ is a $C_0$-semigroup.
Note that in Subsections~\ref{Smax5.2} and \ref{Smax5.1} we could quote the 
literature to have a $C_0$-semigroup on $C_0(\Omega)$ and $C(\overline \Omega)$.

\begin{thm} \label{tmax420}
Adopt the above notation and assumptions.
\begin{tabel} 
\item \label{tmax420-1}
The semigroup $S$ is positive and irreducible.
\item \label{tmax420-2}
If $t > 0$, then $S_t L_2(\Omega) \subset C_D(\overline \Omega)$.
In particular, the semigroup $S$ leaves $C_D(\overline \Omega)$ invariant.
\end{tabel}
For all $t > 0$ define 
$T_t = S_t|_{C_D(\overline \Omega)} \colon C_D(\overline \Omega) \to C_D(\overline \Omega)$.
\begin{tabel} 
\setcounter{teller}{2}
\item \label{tmax420-3}
The semigroup $T$ is a $C_0$-semigroup on $C_D(\overline \Omega)$.
\end{tabel}
\end{thm}
\begin{proof}
`\ref{tmax420-1}'. 
This follows from \cite[Corollary~4.3 and Theorem~4.5]{Ouh5}.

`\ref{tmax420-2}'. 
We first show that $C(\overline \Omega) \cap W^{1,2}_D(\Omega) \subset C_D(\overline \Omega)$.
Let $v \in C(\overline \Omega) \cap W^{1,2}_D(\Omega)$.
Since $(\Tr w)|_D = 0$ $\ch^{d-1}$-almost everywhere for all $w \in C^\infty_D(\Omega)$
it follows by density that 
$(\Tr w)|_D = 0$ $\ch^{d-1}$-almost everywhere for all functions $w \in W^{1,2}_D(\Omega)$.
In particular, $(\Tr v)|_D = 0$ $\ch^{d-1}$-almost everywhere.
Let $z \in D$. 
Suppose that $v(z) \neq 0$.
Since $v$ is continuous, there exists an $s \in (0,1)$ such that 
$v(x) \neq 0$ for all $x \in \overline \Omega \cap B(z,s)$.
If $z$ is in the interior of $D$ in the relative topology of $\partial \Omega$, 
then this contradicts $(\Tr v)|_D = 0$ $\ch^{d-1}$-almost everywhere.
Alternatively, if $z \in \partial D$, then (\ref{etmax420;1}) gives a contradiction.
So $v(z) = 0$.
Therefore $v \in C_D(\overline \Omega)$ and the inclusion
$C(\overline \Omega) \cap W^{1,2}_D(\Omega) \subset C_D(\overline \Omega)$ follows.

It is a consequence of \cite[Theorem~1.1]{ERe2} that $S$ maps into the 
(globally) H\"older continuous functions on $\overline \Omega$.
Let $t > 0$ and $u \in L_2(\Omega)$.
Then $S_t u \in C(\overline \Omega) \cap W^{1,2}_D(\Omega) \subset C_D(\overline \Omega)$
and Statement~\ref{tmax420-2} follows.

`\ref{tmax420-3}'. 
The proof is inspired by the proof of Lemma~4.2 in \cite{Nit4}.
Let $A_{C_D}$ be the part of $A$ in $C_D(\overline \Omega)$.
So
\[
D(A_{C_D}) = \{ u \in C_D(\overline \Omega) \cap D(A) : A u \in C_D(\overline \Omega) \}
\]
and $A_{C_D} u = A u$ for all $u \in D(A_{C_D})$.

First we shall show that $D(A_{C_D})$ is dense in $C_D(\overline \Omega)$.
We shall do this in two steps.
If $u \in C_D(\overline \Omega)^+$ and $\varepsilon > 0$, then 
$D \cap \supp ((u - \varepsilon)^+) = \emptyset$. 
Regularising $(u - \varepsilon)^+$ it follows that 
$C_D(\overline \Omega)^+$ is contained in the $C_D(\overline \Omega)$-closure of 
$ \{ \chi|_{\overline \Omega} : 
      \chi \in C_c^\infty(\Ri^d) \mbox{ and } D \cap \supp \chi = \emptyset \} $.
Hence by linearity 
$ \{ \chi|_{\overline \Omega} : 
      \chi \in C_c^\infty(\Ri^d) \mbox{ and } D \cap \supp \chi = \emptyset \} $
is dense in $C_D(\overline \Omega)$.

By Proposition~\ref{pmax421}\ref{pmax421-2} below
there exists a $c > 0$ such that $u \in C(\overline \Omega)$ and 
$\|u\|_{C(\overline \Omega)} \leq c \sum_{k=1}^d \|f_k\|_{L_{d+1}(\Omega)}$
for all $u \in W^{1,2}_D(\Omega)$ and $f_1,\ldots,f_d \in L_{d+1}(\Omega)$ 
such that
\[
\gota(u,v) = \sum_{k=1}^d (f_k, \partial_k v)_{L_2(\Omega)}
\]
for all $v \in W^{1,2}_D(\Omega)$.
Now let $\chi \in C_c^\infty(\Ri^d)$ and suppose that $D \cap \supp \chi = \emptyset$.
Let $\varepsilon > 0$.
Since $C_c^\infty(\Omega)$ is dense in $L_{d+1}(\Omega)$, 
for all $k \in \{ 1,\ldots,d \} $ there exists a $w_k \in C_c^\infty(\Omega)$ such that 
\[
\left\|w_k - \sum_{l=1}^d a_{lk} \, \partial_l \chi\right\|_{L_{d+1}(\Omega)} < \varepsilon.
\]
Define $f = - \sum_{k=1}^d \partial_k w_k \in C_c^\infty(\Omega)$.
There exists a unique $u \in W^{1,2}_D(\Omega)$ such that 
$\gota(u,v) = (f, v)_{L_2(\Omega)}$ for all $v \in W^{1,2}_D(\Omega)$.
Then $u \in C(\overline \Omega)$ by Proposition~\ref{pmax421}\ref{pmax421-1} below.
So $u \in C(\overline \Omega) \cap W^{1,2}_D(\Omega) \subset C_D(\overline \Omega)$
by the first step in the proof of Statement~\ref{tmax420-2}.
Clearly $u \in D(A)$ and $A u = f$.
Obviously $f \in C_D(\overline \Omega)$.
Hence $u \in D(A_{C_D})$.
Moreover, if $v \in W^{1,2}_D(\Omega)$, then 
\begin{eqnarray*}
\gota(u - \chi|_\Omega, v)
& = & \sum_{k,l=1}^d \int_\Omega a_{kl} \, (\partial_k u - \partial_k \chi) \, \overline{\partial_l v}  \\
& = & (f,v)_{L_2(\Omega)} 
     - \sum_{k,l=1}^d \int_\Omega a_{kl} \, (\partial_k \chi) \, \overline{\partial_l v}  \\
& = & - \sum_{k=1}^d (\partial_k w_k, v)_{L_2(\Omega)}
    - \sum_{k=1}^d \sum_{l=1}^d (a_{lk} \, \partial_l \chi, \partial_k v)_{L_2(\Omega)}  \\
& = & \sum_{k=1}^d (w_k - \sum_{l=1}^d a_{lk} \, \partial_l \chi, \partial_k v)_{L_2(\Omega)} 
.  
\end{eqnarray*}
So $u - \chi|_{\overline \Omega} \in C(\overline \Omega)$ and 
\[
\|u - \chi|_{\overline \Omega}\|_{C(\overline \Omega)} 
   \leq c \sum_{k=1}^d \Big\|w_k - \sum_{l=1}^d a_{lk} \, \partial_l \chi\Big\|_{L_{d+1}(\Omega)}
   \leq c \, d \, \varepsilon.
\]
We showed that $\chi|_{\overline \Omega}$ belongs to the closure of $D(A_{C_D})$
in $C_D(\overline \Omega)$.
Hence $D(A_{C_D})$ is dense in $C_D(\overline \Omega)$.

Now we are able to complete the proof of Statement~\ref{tmax420-3}.
By \cite[Theorem~7]{ERe2}.5 the semigroup $S$ has Gaussian kernel bounds.
Hence there exists an $M > 0$ such that $\|S_t\|_{\infty \to \infty} \leq M$
for all $t \in (0,1]$.
Then $\|T_t\|_{\infty \to \infty} \leq M$ for all $t \in (0,1]$.
If $u \in D(A_{C_D})$, then 
\[
\|(I - T_t)u\|_{C_D(\overline \Omega)}
\leq \int_0^t \|S_s A_{C_D} u\|_\infty
\leq M \, t \, \|A_{C_D} u\|_\infty
\]
for all $t \in (0,1]$.
Hence we have $\lim_{t \downarrow 0} T_t u = u$ in $C_D(\overline \Omega)$.
Since $D(A_{C_D})$ is dense in $C_D(\overline \Omega)$ the semigroup $T$ is a 
$C_0$-semigroup.
\end{proof}

In the proof Theorem~\ref{tmax420} we needed the following regularity results of \cite{ERe2}.

\begin{prop} \label{pmax421}
Let $p \in (d,\infty)$.
\begin{tabel}
\item \label{pmax421-1}
If $u \in W^{1,2}_D(\Omega)$ and $f \in L_p(\Omega)$ with 
$\gota(u,v) = (f,v)_{L_2(\Omega)}$
for all $v \in W^{1,2}_D(\Omega)$, then $u \in C(\overline \Omega)$.
\item \label{pmax421-2}
There exists a constant $c > 0$ such that 
$\|u\|_{C(\overline \Omega)} \leq c \sum_{k=1}^d \|f_k\|_{L_p(\Omega)}$
for all $u \in W^{1,2}_D(\Omega)$ and $f_1,\ldots,f_d \in L_p(\Omega)$ 
such that $\gota(u,v) = \sum_{k=1}^d (f_k, \partial_k v)_{L_2(\Omega)}$
for all $v \in W^{1,2}_D(\Omega)$.
\end{tabel}
\end{prop}
\begin{proof}
This follows as in the proof of Theorem~6.8 in \cite{ERe2}.
Since $\emptyset \neq D \neq \partial \Omega$, the form $\gota$ is coercive.
Hence the identity operator in \cite[Theorem~6.8]{ERe2} is not needed.
\end{proof}

Similar to the case of Dirichlet and Robin boundary conditions, we obtain irreducibility of the
semigroup on the space $C_D(\overline{\Omega})$.

\begin{thm} \label{tmax422}
The operator $A$ on $L_2(\Omega)$ and the semigroup $T$ on $C_D(\overline{\Omega})$
have the following properties.
\begin{tabel} 
\item \label{tmax422-1}
For all $t > 0$ the operator $T_t$ is positivity improving.
In particular, the semigroup $T$ is irreducible.
\item \label{tmax422-2}
Let $u$ be the principal eigenvector of $A$. 
Then $u(x) > 0$ for all $x \in \Omega \cup N$.
\end{tabel}
\end{thm}

\begin{proof}
`\ref{tmax422-1}'. 
Choose $p = 2$.
Let $X = \Omega \cup N$.
Then $\Omega \subset X \subset \overline \Omega$.
It follows from Theorem~\ref{tmax420}\ref{tmax420-2} that 
Condition~\ref{tmax306-c1} in Theorem~\ref{tmax306} is valid and 
Condition~\ref{tmax306-c2} follows from Theorem~\ref{tmax420}\ref{tmax420-3}
for every $x \in X$.
Hence Theorem~\ref{tmax306} implies that $(S_t u)(x) > 0$ 
for all $x \in X$, $t > 0$ and $u \in L_2(\Omega)$ 
with $u \geq 0$ and $u \neq 0$.
So $T$ is positivity improving and consequently irreducible.

`\ref{tmax422-2}'. 
This follows immediately from the proof of Statement~\ref{tmax422-1}.
\end{proof}

\begin{remark} \label{rmax422.5}
Note that $C_D(\overline{\Omega}) = C_0(\Omega \cup N)$,
the closure of $C_c(\Omega \cup N)$ in $C(\overline \Omega)$.
It follows that $u \in C_0(\Omega \cup N)$ is a quasi-interior point 
if and only if $u(x) > 0$ for all
$x \in \Omega \cup N$. 
\end{remark}

By \cite[Theorem~7.5]{ERe2} the semigroup $S$ has Gaussian kernel bounds.
Hence the semigroup extends consistently to $L_p(\Omega)$ for all $p \in [1,\infty)$.
We denote the generator by~$- A_p$.
If $p \in (d/2, \infty)$, then a Laplace transform gives that the 
resolvent of $A_p$ maps $L_p(\Omega)$ into $C_D(\overline \Omega)$.
By the same arguments as in the proof of Corollary~\ref{cmax503} we 
obtain the following consequence of Theorem~\ref{tmax422}.

\begin{cor} \label{cmax5041}
Let $\lambda \in \Ri$ be smaller than the first eigenvalue of $A$ and let $p \in (d/2,\infty)$. 
If $u \in D(A_p)$ and $(-\lambda \, I + A_p)u > 0$, then $u \in C_D(\overline{\Omega})$ and $u(x) > 0$ for all $x \in \Omega \cup N$.
\end{cor}

\section{The strong maximum principle for parabolic equations} \label{Smax5}

In this section we show how our results, in particular Corollary~\ref{cmax302}
and Theorem~\ref{tmax502}, can be employed to prove strong minimum and 
maximum principles for parabolic and elliptic differential operators.
Throughout this section let $\Omega \subset \Ri^d$ be a bounded non-empty open set
with boundary~$\Gamma$.
For all $k,l \in \{ 1,\ldots,d \} $ let $a_{kl}, b_k, c_k, c_0 \in L_\infty(\Omega,\Ri)$.
We assume that there exists a $\mu > 0$ such that 
\[
\RRe \sum_{k,l=1}^d a_{kl}(x) \, \xi_k \, \overline{\xi_l}
\geq \mu \, |\xi|^2
\]
for all $\xi \in \Ci^d$ and almost every $x \in \Omega$.
Define $\ca \colon H^1_\loc(\Omega) \to \cd'(\Omega)$ by 
\[
\langle \ca u,v \rangle_{\cd'(\Omega) \times \cd(\Omega)}
= \sum_{k,l=1}^d \int_\Omega a_{kl} \, (\partial_k u) \, \overline{\partial_l v}
   + \sum_{k=1}^d \int_\Omega b_k \, u \, \overline{\partial_k v}
   + \sum_{k=1}^d \int_\Omega c_k \, (\partial_k u) \, \overline v
   + \int_\Omega c_0 \, u \, \overline v
\]
for all $u \in H^1_\loc(\Omega)$ and $v \in C_c^\infty(\Omega)$.
Define the operator $A_{c,\max}$ in $C(\overline \Omega)$ by
\[
D(A_{c,\max})
= \{ u \in H^1_\loc(\Omega) \cap C(\overline \Omega) : \ca u \in C(\overline \Omega) \}
\]
and $A_{c,\max} = \ca|_{D(A_{c,\max})}$.
In this section we shall prove a maximum principle for parabolic equations
involving the operator $A_{c,\max}$.

The maximum and minimum principles in this section are not completely novel.
For operators in non-divergence form they are classical.
For operators in divergence form as we consider them here, 
there are results in \cite[Theorem~6.25]{Liebe3}, with a slightly different notion 
of solution and domain of the operator.
Still, we find it worthwhile to include this section since it shows that our 
approach from the previous sections yields a new short and elementary 
proof for strong parabolic and elliptic maximum principles under very general assumptions on the coefficients of the operator.

\subsection{The strong maximum principle for mild solutions} \label{Smax5.3}

In this subsection, we assume in addition that $\Omega$ is connected and Wiener regular
(see the beginning of Subsection~\ref{Smax5.2} for a definition). 
Moreover, we assume that the coefficients satisfy
\[
\int_\Omega c_0 \, v + \sum_{k=1}^d \int_\Omega b_k \, \partial_k v
\geq 0
\]
for all $v \in C_c^\infty(\Omega)^+$.
Fix $T \in (0,\infty)$.
Let $\varphi \in C([0,T],C(\Gamma))$ and $u_0 \in C(\overline \Omega)$.
Formally we consider the problem
\begin{equation}
\left[ \begin{array}{ll}
\dot{u}(t) = - A_{c,\max} u(t) \quad & \mbox{for all } t \in [0,T],  \\[5pt]
u(t)|_\Gamma = \varphi(t) &  \mbox{for all } t \in [0,T],  \\[5pt]
u(0) = u_0.
        \end{array} \right.
\label{etmax503;1}
\end{equation}
As in \cite{Are9} we say that $u \in C([0,T],C(\overline \Omega))$ is a 
{\bf mild solution of (\ref{etmax503;1})} if 
\[
\int_0^t u(s) \, ds \in D(A_{c,\max})
\quad , \quad
u(t) = u_0 - A_{c,\max} \int_0^t u(s) \, ds
\quad \mbox{and} \quad
u(t)|_\Gamma = \varphi(t)
\]
for all $t \in [0,T]$.
Arendt \cite[Theorem~6.5]{Are9} proved the following theorem.

\begin{thm} \label{tmax503}
Let $\varphi \in C([0,T],C(\Gamma))$ and $u_0 \in C(\overline \Omega)$
with $u_0|_\Gamma = \varphi(0)$.
Then there exists a unique function $u \in C([0,T],C(\overline \Omega))$
such that $u$ is a mild solution of (\ref{etmax503;1}).
Moreover, if $\varphi \geq 0$ and $u_0 \geq 0$, then $u \geq 0$.
\end{thm}

The last part can be improved with the aid of Corollary~\ref{cmax302}. This is the main
result of this subsection.

\begin{thm} \label{tmax504}
Let $\varphi \in C([0,T],C(\Gamma))$ and $u_0 \in C(\overline \Omega)$
with $u_0|_\Gamma = \varphi(0)$, $\varphi \geq 0$ and $u_0 \geq 0$.
Let $u \in C([0,T],C(\overline \Omega))$ be the mild solution of (\ref{etmax503;1}).
\begin{tabel}
\item \label{tmax504-1}
If $u_0 \neq 0$, then $u(t,x) > 0$ for all $t \in (0,T]$ and $x \in \Omega$.
\item \label{tmax504-2}
If $t_0 \in [0,T)$ and $\varphi(t_0) \neq 0$,
then $u(t,x) > 0$ for all $t \in (t_0,T]$ and $x \in \Omega$.
\end{tabel}
\end{thm}
\begin{proof}
`\ref{tmax504-1}'.
There exists an $x_0 \in \Omega$ such that $u_0(x_0) \neq 0$.
Let $\chi \in C_c^\infty(\Omega)$ be such that $0 \leq \chi \leq \one$
and $\chi(x_0) = 1$.
Consider $v_0 = \chi \, u_0 \in C_0(\Omega)$.
By Theorem~\ref{tmax503} there exists a unique $v \in C([0,T],C(\overline \Omega))$
such that 
\begin{equation}
\int_0^t v(s) \, ds \in D(A_{c,\max})
\quad , \quad
v(t) = v_0 - A_{c,\max} \int_0^t v(s) \, ds
\quad \mbox{and} \quad
v(t)|_\Gamma = 0
\label{etmax504;1}
\end{equation}
for all $t \in [0,T]$.
The function $v$ can be described via a semigroup.
Let $A_c$ be the part of $A_{c,\max}$ in $C_0(\Omega)$.
So 
\[
D(A_c) 
= \{ u \in H^1_\loc(\Omega) \cap C_0(\Omega) : \ca u \in C_0(\Omega) \}
\]
and $A_c = \ca|_{D(A_c)}$.
Then $- A_c$ generates a $C_0$-semigroup on $C_0(\Omega)$ by \cite[Theorem~1.3]{AE9}
(see also Subsection~\ref{Smax5.2} or \cite[Section~4]{ABenilan}).
Let $T$ be the semigroup generated by $-A_c$.
Then $T$ is positive and irreducible by Theorem~\ref{tmax502}.
Define $w \colon [0,T] \to C(\overline \Omega)$ by 
$w(t) = T_t v_0$.
Then it is easy to see that $w \in C([0,T],C(\overline \Omega))$ and that 
$w$ satisfies (\ref{etmax504;1}) with $v$ replaced by $w$.
So $v(t) = w(t) = T_t v_0$ for all $t \in (0,T]$ by the uniqueness property.
The semigroup $T$ also extends to a positive irreducible holomorphic 
$C_0$-semigroup on $L_d(\Omega)$ and this semigroup maps 
$L_d(\Omega)$ into $C_0(\Omega)$.
Hence we can apply Corollary~\ref{cmax302}\ref{cmax302-1} and 
conclude that $v(t,x) > 0$ for all $t \in (0,T]$ and $x \in \Omega$.

Finally consider $u-v \in C([0,T],C(\overline \Omega))$.
Then
\begin{align*}
& \int_0^t (u-v)(s) \, ds \in D(A_{c,\max}), \\
& (u-v)(t) = (u_0-v_0) - A_{c,\max} \int_0^t (u-v)(s) \, ds \\
\text{and} \quad & (u-v)(t)|_\Gamma = \varphi(t)
\end{align*}
for all $t \in [0,T]$.
So $u-v \geq 0$ by the last part of Theorem~\ref{tmax503}.
In particular, $u(t,x) \geq v(t,x) > 0$ for all $t \in (0,T]$ and 
$x \in \Omega$.

`\ref{tmax504-2}'.
Apply Statement~\ref{tmax504-1} to $\tilde u_0 = u(t_0,\cdot)$, 
$\tilde u(t,x) = u(t - t_0,x)$ and $\tilde \varphi(t) = \varphi(t-t_0)$, 
where $t \in [0,T-t_0]$ and $x \in \overline \Omega$.
\end{proof}

In the following corollary we show how a strong parabolic maximum principle can 
be derived from Theorem~\ref{tmax504}.

\begin{cor} \label{cmax508}
Assume that $\ca \one_\Omega = 0$.
Let $\varphi \in C([0,T],C(\Gamma))$ and $u_0 \in C(\overline \Omega)$.
Let $u \in C([0,T],C(\overline \Omega))$ be the mild solution of (\ref{etmax503;1}).
Moreover, let $t_0 \in (0,T]$ and $x_0 \in \Omega$.
If $u(t_0,x_0) \geq u(t,x)$ for all $t \in [0,t_0]$ and $x \in \overline \Omega$, then
$u$ is constant on $[0,t_0] \times \overline \Omega$.
\end{cor}
\begin{proof}
Define $v \in C([0,t_0],C(\overline \Omega))$ by 
$v(t,x) = u(t_0,x_0) - u(t,x)$.
Define $v_0 \in C(\overline \Omega)$ by $v_0(x) = u(t_0,x_0) - u_0(x)$
and define $\psi \in C([0,t_0],C(\Gamma))$ by 
$\psi(t,x) = u(t_0,x_0) - \varphi(t,x)$.
Then $v_0|_\Gamma = \psi(0)$ and $v \geq 0$.
So $\psi \geq 0$ and $v_0 \geq 0$.
Also 
\[
\int_0^t v(s) \, ds \in D(A_{c,\max})
\quad , \quad
v(t) = v_0 - A_{c,\max} \int_0^t v(s) \, ds
\quad \mbox{and} \quad
v(t)|_\Gamma = \psi(t)
\]
for all $t \in [0,t_0]$.
Since $v(t_0,x_0) = 0$, it follows from Theorem~\ref{tmax504}
that $v_0 = 0$ and $\psi = 0$.
Then the uniqueness part of Theorem~\ref{tmax503} implies that $v = 0$.
Hence $u$ is constant on $[0,t_0] \times \overline \Omega$.
\end{proof}

In the next two corollaries we deduces a strong elliptic 
maximum principle from the parabolic result in Corollary~\ref{cmax508}.

\begin{cor} \label{cmax505}
Let $u \in H^1_\loc(\Omega) \cap C(\overline \Omega)$ and suppose that 
$\ca u = 0$.
If $u \geq 0$ and $u|_\Gamma \neq 0$, then $u(x) > 0$ for all $x \in \Omega$.
\end{cor}
\begin{proof}
Define $v_0 = u$, $\varphi(t) = u|_\Gamma$ and $v(t) = u$
for all $t \in [0,T]$.
Then $v$ is a mild solution of (\ref{etmax503;1}) with $u_0$ replaced by $v_0$.
Now apply Theorem~\ref{tmax504}\ref{tmax504-1}.
\end{proof}

\begin{cor} \label{cmax506}
Suppose that $\ca \one_\Omega = 0$.
Let $u \in H^1_\loc(\Omega) \cap C(\overline \Omega,\Ri)$ and suppose that 
$\ca u = 0$.
If there there exists an $x_0 \in \Omega$ such that 
$u(x_0) = \max_\Omega u$, then $u$ is constant.
\end{cor}
\begin{proof}
Consider $v = u(x_0) \, \one_{\overline \Omega} - u$.
Then $v \geq 0$ and $\ca v = 0$.
Since $v(x_0) = 0$, it follows from Corollary~\ref{cmax505}
that $v = 0$.
Hence $u = u(x_0) \, \one_{\overline \Omega}$ and $u$ is constant.
\end{proof}

\subsection{Mild and very weak solutions} \label{Smax5.6}

Theorem~\ref{tmax504} and the parabolic maximum principle 
in Corollary~\ref{cmax508} used the concept of a mild solution of (\ref{etmax503;1}).
In this subsection we show
under a differentiability condition that mild 
solutions are the same as very weak solutions.

Throughout this subsection we assume that the coefficients $a_{kl}, c_k$ satisfy
$a_{kl}, c_k \in L_\infty(\Omega,\Ri) \cap C^1(\Omega)$ for all $k,l \in \{ 1,\ldots,d \}$.

Fix a time $T \in (0,\infty)$. For all $u \in C([0,T] \times \overline \Omega)$ define 
$\tilde u \in C([0,T],C(\overline \Omega))$ by 
$(\tilde u(t))(x) = u(t,x)$.
If $\psi \in C_c^\infty(\Omega)$ define $\ca^* \psi \in C_c(\Omega)$
by 
\[
\ca^* \psi 
= - \sum_{k,l=1}^d \partial_k  ( a_{kl} \, \partial_l \psi)
   + \sum_{k=1}^d b_k \, \partial_k \psi
   - \sum_{k=1}^d \partial_k ( c_k \, \psi)
   + c_0 \, \psi
.  \]
If $\varphi \in C_c^\infty((0,T) \times \Omega)$ define 
$\ca^* \varphi \in C([0,T] \times \overline \Omega)$ by 
\[
(\ca^* \varphi)(t,x) 
= \Big( \ca^*(\tilde \varphi(t)) \Big)(x)
.  \]
Moreover, we define $\varphi_t \in C_c^\infty((0,T) \times \Omega)$ by 
\[
\varphi_t(t,x) = (\frac{\partial}{\partial t} \varphi)(t,x).
\]

The second condition in the next theorem states that $u$ is a {\it very weak solution}.

\begin{thm} \label{tmax601}
Let $u \in C([0,T] \times \overline \Omega)$.
Then the following are equivalent.
\begin{tabeleq}
\item \label{tmax601-1}
$\int_0^t \tilde u(s) \, ds \in D(A_{c,\max})$ and 
$\tilde u(t) = \tilde u(0) - A_{c,\max} \int_0^t \tilde u(s) \, ds$
for all $t \in [0,T]$.
\item \label{tmax601-2}
If $\varphi \in C_c^\infty((0,T) \times \Omega)$, then 
$\int_0^T \int_\Omega u(t,x) \, (\varphi_t - \ca^* \varphi)(t,x) \, dx \, dt = 0$.
\end{tabeleq}
\end{thm}
\begin{proof}
`\ref{tmax601-1}$\Rightarrow$\ref{tmax601-2}'.
Let $\varphi \in C_c^\infty((0,T) \times \Omega)$.
Define $v \in C([0,T] \times \overline \Omega)$ by 
$v(t,x) = \int_0^t u(s,x) \, ds$.
Then $\tilde v(t) = \int_0^t \tilde u(s) \, ds$ for all $t \in [0,T]$.
Therefore
\begin{eqnarray*}
\int_0^T \int_\Omega u(t,x) \, (\ca^* \varphi)(t,x) \, dx \, dt
& = & \int_\Omega \int_0^T \Big( \tfrac{\partial}{\partial t} v(t,x) \Big)
            \, (\ca^* \varphi)(t,x) \, dt \, dx  \\
& = & - \int_\Omega \int_0^T v(t,x) \, (\ca^* \varphi_t)(t,x) \, dt \, dx  \\
& = & - \int_0^T (\tilde v(t), \ca^* \overline{\tilde \varphi_t(t)})_{L_2(\Omega)} \, dt  \\
& = & - \int_0^T (A_{c,\max} \tilde v(t), \overline{\tilde \varphi_t(t)})_{L_2(\Omega)} \, dt  \\
& = & \int_0^T (\tilde u(t) - \tilde u(0), \overline{\tilde \varphi_t(t)})_{L_2(\Omega)} \, dt  \\
& = & \int_0^T \int_\Omega u(t,x) \, \varphi_t(t,x) \, dx \, dt 
   - \int_\Omega u(0,x) \int_0^T \varphi_t(t,x) \, dt \, dx .
\end{eqnarray*}
Since $\int_0^T \varphi_t(t,x) \, dt = \varphi(T,x) - \varphi(0,x) = 0$
for all $x \in \Omega$, the implication follows.

`\ref{tmax601-2}$\Rightarrow$\ref{tmax601-1}'.
Let $\psi \in C_c^\infty(\Omega)$,
and define $f,g \in C[0,T]$ by 
$f(t) = (\tilde u(t), \psi)_{L_2(\Omega)}$ and
$g(t) = (\tilde u(t), \ca^* \psi)_{L_2(\Omega)}$.
Let $\tau \in C_c^\infty((0,T))$.
Then $\tau \otimes \overline \psi \in C_c^\infty((0,T) \times \Omega)$.
So by assumption
\[\int_0^T \int_\Omega 
    u(t,x) \, ((\tau \otimes \overline \psi)_t - \ca^* (\tau \otimes \overline \psi))(t,x) \, dx \, dt = 0.
\]
Hence 
\begin{eqnarray*}
\int_0^T \tau'(t) \, f(t) \, dt
& = & \int_0^T \int_\Omega u(t,x) \, \tau'(t) \, \overline \psi(x) \, dx \, dt  \\
& = & \int_0^T \int_\Omega u(t,x) \, \tau(t) \, (\ca^* \overline \psi)(x) \, dx \, dt  \\
& = & \int_0^T \tau(t) \, g(t) \, dt
\end{eqnarray*}
and $f' = - g$ weakly.
Since $f$ and $g$ are continuous the lemma of du Bois-Reymond implies that 
$f$ is differentiable and $f' = - g$ in the classical sense.
Let $t \in [0,T]$.
Then 
\begin{eqnarray*}
(\tilde u(t), \psi)_{L_2(\Omega)}
& = & f(t)
= f(0) - \int_0^t g(s) \, ds
= (\tilde u(0), \psi)_{L_2(\Omega)} 
   - \int_0^t (\tilde u(s), \ca^* \psi)_{L_2(\Omega)} \, ds
.  
\end{eqnarray*}
So 
\[
\Big( \int_0^t \tilde u(s) \, ds, \ca^* \psi \Big)_{L_2(\Omega)}
= ( \tilde u(0) - \tilde u(t), \psi)_{L_2(\Omega)}
.  \]
This is for all $\psi \in C_c^\infty(\Omega)$.
It follows from elliptic regularity, see Proposition~\ref{pmaxapp1} in the appendix, that 
$\int_0^t \tilde u(s) \, ds \in H^1_\loc(\Omega)$.
Hence~\ref{tmax601-1} is valid.
\end{proof}

Now we can reformulate the results of the previous subsection
using the notion of very weak solutions.
We consider the parabolic cylinder $\Omega_T = (0,T) \times \Omega$
with parabolic boundary 
$\partial^* \Omega_T 
   = \Big( \{ 0 \} \times \overline \Omega \Big) \cup 
     \Big( (0,T] \times \partial \Omega \Big)$.
Given $\psi \in C(\partial^* \Omega_T)$ we formally
consider the Dirichlet problem for the heat equation
\[
D(\psi)
\hspace*{20mm}
\left[ \begin{array}{ll}
   u \in C(\overline{\Omega_T}) ,  \\[5pt]
   \partial_t u - \ca u = 0 & \mbox{on } \Omega_T ,  \\[5pt]
   u|_{\partial^* \Omega_T} = \psi .
       \end{array} \right.
\]
We say that $u \in C(\overline{\Omega_T})$ is a {\bf very weak solution of $D(\psi)$}
if 
\[
\int_0^T \int_\Omega u(t,x) \, (\varphi_t - \ca^* \varphi)(t,x) \, dx \, dt = 0
\]
for all $\varphi \in C_c^\infty((0,T) \times \Omega)$ and 
$u|_{\partial^* \Omega_T} = \psi$.
Then the following holds.

\begin{thm} \label{tmax602}
Assume $\Omega$ is connected and Wiener regular.
Then for all 
$\psi \in C(\partial^* \Omega_T)$ there exists a unique very weak solution 
of $D(\psi)$.
\end{thm}

For this solution of $D(\psi)$ the following maximum principles are valid.

\begin{thm} \label{tmax603}
Assume $\Omega$ is connected and Wiener regular.
Let $\psi \in C(\partial^*_T \Omega)$ and let 
$u \in C(\overline{\Omega_T})$ be the very weak solution $D(\psi)$.
Then one has the following.
\begin{tabel} 
\item \label{tmax603-1}
If $\psi \geq 0$, then $u \geq 0$.
\item \label{tmax603-2}
Let $t_0 \in [0,T)$ and $x_0 \in \overline \Omega$.
Suppose that $u(t_0,x_0) > 0$ and $\psi \geq 0$.
Then $u(t,x) > 0$ for all $t \in (t_0,T]$ and $x \in \Omega$.
\item \label{tmax603-3}
Suppose $\ca \one_\Omega = 0$.
Let $t_0 \in (0,T]$ and $x_0 \in \Omega$.
If $u(t,x) \leq u(t_0,x_0)$ for all $t \in [0,t_0]$ and $x \in \overline \Omega$,
then $u$ is constant on $[0,t_0] \times \overline \Omega$.
\end{tabel}
\end{thm}

The maximum principle for elliptic operators has been proved before 
in \cite[Theorem~8.19]{GT}.

\appendix

\section{Regularity} \label{Smaxapp}

In the proof of Theorem~\ref{tmax601} we used the following regularity result
for very weak solutions.

\begin{prop} \label{pmaxapp1}
Let $\Omega \subset \Ri^d$ be open.
For all $k,l \in \{ 1,\ldots,d \} $ let 
$a_{kl}, c_k \in W^{1,\infty}(\Omega)$ and 
$b_k,c_0 \in L_\infty(\Omega)$.
Assume that there exists a $\mu > 0$ such that 
\begin{equation}
\RRe \sum_{k,l=1}^d a_{kl}(x) \, \xi_k \, \overline{\xi_l}
\geq \mu \, |\xi|^2
\label{epmaxapp1;1}
\end{equation}
for all $\xi \in \Ci^d$ and almost every $x \in \Omega$.
For all  $\varphi \in C_c^\infty(\Omega)$ define $\ca^* \varphi \in L_{\infty,c}(\Omega)$
by 
\[
\ca^* \varphi 
= - \sum_{k,l=1}^d \partial_k  ( \overline{a_{kl}} \, \partial_l \varphi)
   + \sum_{k=1}^d \overline{b_k} \, \partial_k \varphi
   - \sum_{k=1}^d \partial_k ( \overline{c_k} \, \varphi)
   + \overline{c_0} \, \varphi
.  \]
Let $u,f,f_1,\ldots,f_d \in L_2(\Omega)$ and suppose that 
\[
(u,\ca^* \varphi)_{L_2(\Omega)} 
= (f,\varphi)_{L_2(\Omega)} - \sum_{j=1}^d (f_j, \partial_j \varphi)_{L_2(\Omega)}
\]
for all $\varphi \in C_c^\infty(\Omega)$.
Then $u \in W^{1,2}_\loc(\Omega)$.
\end{prop}
\begin{proof}
Fix $\chi \in C_c^\infty(\Omega)$.
We shall show that $\chi \, u \in W^{1,2}(\Omega)$.
Without loss of generality we may assume that 
$a_{kl}, c_k \in W^{1,\infty}(\Ri^d)$ and 
$b_k,c_0 \in L_\infty(\Ri^d)$, and that (\ref{epmaxapp1;1}) is valid for all 
$\xi \in \Ci^d$ and almost every $x \in \Ri^d$.
Define the form $\gota \colon W^{1,2}(\Ri^d) \times W^{1,2}(\Ri^d) \to \Ci$ by 
\[
\gota(v,w) 
= \int_{\Ri^d} \sum_{k,l=1}^d a_{kl} \, (\partial_k v) \, \overline{\partial_l w} 
   + \int_{\Ri^d} \sum_{k=1}^d b_k \, v \, \overline{\partial_k w} 
   + \int_{\Ri^d} \sum_{k=1}^d c_k \, (\partial_k v) \, \overline w
   + \int_{\Ri^d} c_0 \, v \, \overline w
.  \]
Then $\gota$ is a closed sectorial form.
Let $A$ be the m-sectorial operator associated with $\gota$.
Note that we have $C_c^\infty(\Ri^d) \subset D(A^*)$ and 
that $A^* \varphi = \ca^* \varphi$ for all $\varphi \in C_c^\infty(\Omega)$.

Define 
\[
g = \chi \, f - \sum_{j=1}^d (\partial_j \chi) \, f_j
   + \sum_{k,l=1}^d u \, \partial_k (a_{kl} \, \partial_l \chi)
   - \sum_{k=1}^d b_k \, u \, \partial_k \chi
   + \sum_{k=1}^d c_k \, u \, \partial_k \chi
\]
and for all $j \in \{ 1,\ldots,d \} $ define 
\[
g_j = \chi \, f_j - \sum_{l=1}^d a_{jl} \, u \, \partial_l \chi
         - \sum_{k=1}^d a_{kj} \, u \, \partial_k \chi.
\]
Then $g,g_2 \in L_{2,c}(\Omega) \subset L_2(\Ri^d)$.
It is a tedious elementary exercise to show that 
\begin{equation}
(\chi \, u,A^* \varphi)_{L_2(\Ri^d)} 
= (g,\varphi)_{L_2(\Ri^d)} - \sum_{j=1}^d (g_j, \partial_j \varphi)_{L_2(\Ri^d)}
\label{epmaxapp1;2}
\end{equation}
for all $\varphi \in C_c^\infty(\Ri^d)$.
It follows from \cite[Theorem~9.8]{Agmrev} that $C_c^\infty(\Ri^d)$ is a 
core for $A^*$. Obviously $D(A^*) \subset W^{1,2}(\Ri^d)$.
Hence (\ref{epmaxapp1;2}) is valid for all $\varphi \in D(A^*)$.

Without loss of generality we may assume
that $\RRe c_0$ is large enough such that $A^*$ is invertible.
Then 
\[
(\chi \, u,w)_{L_2(\Ri^d)} 
= (g,(A^*)^{-1} w)_{L_2(\Ri^d)} 
    - \sum_{j=1}^d (g_j, \partial_j \, (A^*)^{-1} w)_{L_2(\Ri^d)}
\]
for all $w \in L_2(\Ri^d)$.
Let $m \in \{ 1,\ldots,d \} $.
Then 
\[
(\chi \, u, \partial_m v)_{L_2(\Ri^d)} 
= (g,(A^*)^{-1} \, \partial_m v)_{L_2(\Ri^d)} 
   - \sum_{j=1}^d (g_j, \partial_j \, (A^*)^{-1} \, \partial_m v)_{L_2(\Ri^d)}
\]
for all $v \in W^{1,2}(\Ri^d)$.
It follows from the ellipticity condition that the operator
\[
\partial_j \, (A^*)^{-1} \, \partial_m
\]
from $W^{1,2}(\Ri^d)$ into $L_2(\Ri^d)$ has a bounded extension 
from $L_2(\Ri^d)$ into $L_2(\Ri^d)$
for all $j \in \{ 1,\ldots,d \} $.
Since $D(A) \subset W^{1,2}(\Ri^d)$, it follows by duality 
that there is an $M > 0$ such that 
$|(\chi \, u, (A^*)^{-1} \, \partial_m v)_{L_2(\Ri^d)}| \leq M \, \|v\|_{L_2(\Ri^d)}$
for all $v \in W^{1,2}(\Ri^d)$.
Hence $\chi \, u \in W^{1,2}(\Ri^d)$, as required.
\end{proof}

We emphasise that all the coefficients of the operator in Proposition~\ref{pmaxapp1}
may be complex valued, including the second-order coefficients.

\begin{remark} \label{rmaxapp2}
Suppose in addition to the conditions of the coefficients
in Proposition~\ref{pmaxapp1} that $b_k \in W^{1,\infty}(\Ri^d)$
for all $k \in \{ 1,\ldots,d \} $.
Let $p \in (1,\infty)$, $u,f,f_1,\ldots,f_d \in L_p(\Omega)$ and suppose
\[
\int_\Omega u \, \overline{\ca^* \varphi}
= \int_\Omega f \, \overline \varphi 
    - \sum_{j=1}^d \int_\Omega f_j \, \overline{\partial_j \varphi}
\]
for all $\varphi \in C_c^\infty(\Omega)$.
Then $u \in W^{1,p}_\loc(\Omega)$.
The proof is almost the same.
The operator $A$ is consistent with an operator $A_p$ in $L_p(\Ri^d)$
such that the semigroups generated by $-A$ and $-A_p$ are consistent.
Let $q \in (1,\infty)$ be the dual exponent of $p$.
The inclusions $D(A_p^*) \subset W^{1,q}(\Ri^d)$ 
and $D(A_p) \subset W^{1,p}(\Ri^d)$ are in \cite[Corollary 3.8]{ER19}.
The bounded extension of $\partial_j \, (A_p^*)^{-1} \, \partial_m$
follows from \cite[Theorem 1.4]{ER19}.

If in addition $f_1 = \ldots = f_d = 0$, then 
one can deduce as in the proof of Proposition~\ref{pmaxapp1}
that $\chi \, u \in D(A_p^{**})$.
Since $D(A_p) = W^{2,p}(\Ri^d)$ (see \cite[Proposition~5.1]{ER19}),
one establishes that $u \in W^{2,p}_\loc(\Omega)$.
\end{remark}

For real coefficients a slightly more generally version of Remark~\ref{rmaxapp2}
has been proved by Zhang and Bao in~\cite[Theorem~1.5]{ZhangBao2}, where
$f$ is even allowed to be an element of a larger $L_q$-space if $d \geq 3$
and a Lorentz space if $d = 2$. We refer to \cite{ZhangBao2} for an account on known 
regularity results for very weak solutions of elliptic operators.

\subsection*{Acknowledgement:} 
We are grateful to Daniel Daners for illuminating discussions 
on semilinear problems, the weak Harnack inequality and Hopf's maximum principle.
We wish to thank the referee for the comments and suggestions.
The second-named author is most grateful for the hospitality extended
to him during a fruitful stay at the University of Ulm.
He wishes to thank the University of Ulm for financial support.

\subsection*{Funding:} 
Part of this work is supported by the Marsden Fund Council from Government funding,
administered by the Royal Society of New Zealand.

\end{document}